\documentclass{amsproc}%
\usepackage{amsfonts}
\usepackage{amsmath}
\usepackage{amssymb}
\usepackage{graphicx}%
\providecommand{\U}[1]{\protect\rule{.1in}{.1in}}
\theoremstyle{plain}

\newtheorem{corollary}{Corollary}

\newtheorem{definition}{Definition}

\newtheorem{lemma}{Lemma}

\newtheorem{proposition}{Proposition}
\newtheorem{remark}{Remark}

\newtheorem{theorem}{Theorem}
\numberwithin{equation}{section}
\begin{document}
\title[Topological Orbit Dimension in C*-algebras]{Topological Orbit Dimension of MF C*-algebras }
\author{Qihui Li}
\address{Department of Mathematics, East China University of Science and Technology,
Meilong Road 130, 200237 Shanghai, China.}
\email{qihui\_li@126.com}
\author{Don Hadwin}
\address{Department of Mathematics, University of New Hampshire, Durham, NH 03824, USA.}
\email{don@unh.edu}
\author{Weihua Li}
\address{Department of Science and Mathematics, Columbia College of Chicago, Chicago,
IL, U.S.A.}
\email{wli@colum.edu}
\author{Junhao Shen}
\address{Department of Mathematics, University of New Hampshire, Durham, NH 03824, USA.}
\email{junhao.shen@unh.edu}
\thanks{The research of the first author is partially supported by National Natural
Science Foundation of China (11671133).}
\subjclass[2000]{Primary 46L10; Secondary 46L54}
\keywords{Topological free entropy dimension, C*-algebra, Topological orbit dimension}

\begin{abstract}
This paper is a continuation of our work on D. Voiculescu's topological free
entropy dimension $\delta_{\text{top}}\left(  x_{1},\ldots,x_{n}\right)  $ for
$\vec{x}=\left(  x_{1},\ldots,x_{n}\right)  $ of elements in a unital
C*-algebra. In this paper we first prove that $\delta_{top}\left(
x_{1},\cdots,x_{n}\right)  =1-\frac{1}{\dim\mathcal{A}}$ in which C*$\left(
\vec{x}\right)  $ is MF-nuclear and inner QD. Then we give a relation between
the topological orbit dimension $\mathfrak{K}_{top}^{(2)}$ and the modified
free orbit dimension $\mathfrak{K}_{2}^{(2)}$ by using MF-traces. We also
introduce a new invariant $\mathfrak{K}_{top}^{(3)}$ which is a modification
of the topological orbit dimension $\mathfrak{K}_{top}^{(2)}$ when
$\mathfrak{K}_{top}^{(2)}$ is defined. As the applications of $\mathfrak{K}%
_{top}^{(3)},$ We prove that $\mathfrak{K}_{top}^{(3)}\left(  \mathcal{A}%
\right)  =0$ if $\mathcal{A}$ has property c*-$\Gamma$ and has no
finite-dimensional representations. We also give the definition of property
MF-c*-$\Gamma.$ We then conclude that, for the unital MF C*-algebra
$\mathcal{A}=C^{\ast}\left(  x_{1},x_{2},\ldots,x_{n}\right)  \ $with no
finite-dimensional representations, if $\mathcal{A}$ has property
MF-c*-$\Gamma$, then $\mathfrak{K}_{top}^{\left(  3\right)  }\left(
\mathcal{A}\right)  =0.$

\end{abstract}
\maketitle

\section{Introduction}

This paper is a continuation of the work in \cite{HLS}, \cite{HS2},
\cite{HLSW} on D. Voiculescu's topological free entropy dimension
$\delta_{\text{top}}\left(  x_{1},\ldots,x_{n}\right)  $ for an $n$-tuple
$\vec{x}=\left(  x_{1},\ldots,x_{n}\right)  $ of elements in a unital
C*-algebra. Here we first give a relation between the topological orbit
dimension $\mathfrak{K}_{top}^{(2)}$ and the modified free orbit dimension
$\mathfrak{K}_{2}^{(2)}$ by using MF-traces. This result allows us to give a
new proof of our main result in \cite{HLSW}, which gave an estimation of the
upper bound of topological free entropy dimension for MF-nuclear algebras. In
\cite{HLSW}, we have shown that $\delta_{top}\left(  x_{1},\cdots
,x_{n}\right)  =1-\frac{1}{\dim\mathcal{A}}$ if C*$\left(  x_{1},\cdots
,x_{n}\right)  $ is MF-nuclear and residually finite-dimensional (RFD). In
\cite{BK2}, Blackadar and Kirchberg proved that all RFD C*-algebras are inner
quasidiagonal (inner QD). So it is natural to ask whether the topological free
entropy dimension of an MF-nuclear and inner QD algebra is unrelated to its
generating family. In this paper we prove that $\delta_{top}\left(
x_{1},\cdots,x_{n}\right)  =1-\frac{1}{\dim\mathcal{A}}$ if C*$\left(
x_{1},\cdots,x_{n}\right)  $ is MF-nuclear and inner QD$.$

In this article, we also introduce a new invariant $\mathfrak{K}_{top}^{(3)}$
which is a modification of the topological orbit dimension $\mathfrak{K}%
_{top}^{(2)}$ when $\mathfrak{K}_{top}^{(2)}$ is defined. The idea for
defining $\mathfrak{K}_{top}^{(3)}$ arise from the concept $\mathfrak{K}_{3}$
in \cite{HW}. We then extend the domain of $\mathfrak{K}_{top}^{(3)}$ to all
MF algebras and prove that $\mathfrak{K}_{top}^{(3)}$ is a C*-algebra
invariant. We also modify the notion $\mathfrak{K}_{3}$ in \cite{HW} by using
the modified free orbit dimension $\mathfrak{K}_{2}^{(2)}$ and denote it by
$\mathfrak{K}_{3}^{(3)}.$ We give a relation between $\mathfrak{K}_{top}%
^{(3)}$ and $\mathfrak{K}_{3}^{(3)}$ for every MF algebra by using the
relation between the topological orbit dimension $\mathfrak{K}_{top}^{(2)}$
and the modified free orbit dimension $\mathfrak{K}_{2}^{(2)}.$ Several
properties of $\mathfrak{K}_{top}^{(3)}$ are given as follows:

\begin{enumerate}
\item $\mathfrak{K}_{top}^{(3)}\left(  \mathcal{N}_{1}\right)  =\mathfrak{K}%
_{top}^{(3)}\left(  \mathcal{N}_{2}\right)  $ if $C^{\ast}\left(
\mathcal{N}_{1}\right)  =C^{\ast}\left(  \mathcal{N}_{2}\right)  .$

\item If $\mathcal{A}$ is finite generated, then $\mathfrak{K}_{top}^{\left(
3\right)  }\left(  \mathcal{A}\right)  =0$ if $\mathfrak{K}_{top}^{\left(
2\right)  }\left(  \mathcal{A}\right)  =0.$

\item If $\mathcal{N}_{1}\cap\mathcal{N}_{2}$ has no finite-dimensional
representation , then
\[
\mathfrak{K}_{top}^{(3)}\left(  C^{\ast}\left(  \mathcal{N}_{1}\cup
\mathcal{N}_{2}\right)  \right)  \leq\mathfrak{K}_{top}^{(3)}\left(
\mathcal{N}_{1}\right)  +\mathfrak{K}_{top}^{(3)}\left(  \mathcal{N}%
_{2}\right)  .
\]

\item If $\mathcal{N}$ is an MF C*-algebra and $\mathcal{A\subseteq N}$ is a
C*-subalgebra with no finite-dimensional representation$.$ If there is an
unitary $u\in\mathcal{N}$ such that $u\mathcal{A}u^{\ast}\subseteq\mathcal{A}%
$. Then
\[
\mathfrak{K}_{top}^{(3)}\left(  C^{\ast}\left(  \mathcal{A\cup}\left\{
u\right\}  \right)  \right)  \leq\mathfrak{K}_{top}^{(3)}\left(
\mathcal{A}\right)  .
\]

\end{enumerate}

As an application, we prove that $\mathfrak{K}_{top}^{(3)}\left(
\mathcal{A}\right)  =0$ if $\mathcal{A}$ has property c*-$\Gamma$ and has no
finite-dimensional representations. We also give the definition of property
MF-c*-$\Gamma.$ We then conclude that, for the unital MF C*-algebra
$\mathcal{A}=C^{\ast}\left(  x_{1},x_{2},\ldots,x_{n}\right)  \ $with no
finite-dimensional representations, if $\mathcal{A}$ has property
MF-c*-$\Gamma$, then $\mathfrak{K}_{top}^{\left(  3\right)  }\left(
\mathcal{A}\right)  =0.$

The organization of the paper is as follows. In section 2, we recall the
definition of topological free entropy dimension $\delta_{top}\left(
x_{1},\cdots,x_{n}\right)  $ and topological orbit dimension $\mathfrak{K}%
_{top}^{\left(  2\right)  }\left(  x_{1},\cdots,x_{n}\right)  $ of $n$-tuple
$\left(  x_{1},\cdots,x_{n}\right)  $ of elements in a unital C*-algebra. In
section 3, we first give a relation between $\mathfrak{K}_{top}^{\left(
2\right)  }\left(  x_{1},\cdots,x_{n}\right)  $ and $\sup_{\tau\in
\mathcal{T}_{MF}\left(  \mathcal{A}\right)  }\mathfrak{K}_{2}^{\left(
2\right)  }\left(  x_{1},\cdots,x_{n},\tau\right)  $ where $\mathcal{T}%
_{MF}\left(  \mathcal{A}\right)  $ is the set of all MF-traces on C*-algebra
$\mathcal{A}=C^{\ast}\left(  x_{1},\cdots,x_{n}\right)  .$ Then we give a new
proof of our main result in \cite{HLSW}. In section 4, we discuss the
topological free entropy dimension of MF-nuclear and inner QD C*-algebra, we
show that $\delta_{top}\left(  x_{1},\cdots,x_{n}\right)  =1-\frac
{1}{dim\mathcal{A}}$ as $\mathcal{A}=C^{\ast}\left(  x_{1},\cdots
,x_{n}\right)  $ MF-nuclear and inner QD. We introduce topological orbit
dimension $\mathfrak{K}_{top}^{\left(  3\right)  }$ for general MF-algebras in
section 5. Several properties of $\mathfrak{K}_{top}^{\left(  3\right)  }$ are
discussed there. Section 6 is focus on the applications of $\mathfrak{K}%
_{top}^{\left(  3\right)  }$ in central sequence algebras. We prove that
$\mathfrak{K}_{top}^{(3)}\left(  \mathcal{A}\right)  =0$ if $\mathcal{A}$ has
property c*-$\Gamma$ and has no finite-dimensional representations. We also
give the definition of property MF-c*-$\Gamma.$ We then conclude that, for the
unital MF C*-algebra $\mathcal{A}=C^{\ast}\left(  x_{1},x_{2},\ldots
,x_{n}\right)  \ $with no finite-dimensional representations, if $\mathcal{A}$
has property MF-c*-$\Gamma$, then $\mathfrak{K}_{top}^{\left(  3\right)
}\left(  \mathcal{A}\right)  =0.$

\section{Definitions and Preliminaries}

In this section, we are going to recall Voiculescu's definition of the
topological free entropy dimension and topological orbit dimension of
$n$-tuples of elements in a unital C*-algebra.

\subsection{A Covering of a set in a metric space}

Suppose $(X,d)$ is a metric space and $K$ is a subset of $X.$ A family of
balls in $X$ is called a covering of $K$ if the union of these balls covers
$K$ and the centers of these balls lie in $K.$

\subsection{Covering numbers in complex matrix algebra $(\mathcal{M}%
_{k}(\mathbb{C}))^{n}$}

Let $\mathcal{M}_{k}(\mathbb{C})$ be the $k\times k$ full matrix algebra with
entries in $\mathbb{C},$ and $\tau_{k}$ be the normalized trace on
$\mathcal{M}_{k}(\mathbb{C}),$ i.e., $\tau_{k}=\frac{1}{k}Tr,$ where $Tr$ is
the usual trace on $\mathcal{M}_{k}(\mathbb{C}).$ Let $\mathcal{U}(k)$ denote
the group of all unitary matrices in $\mathcal{M}_{k}(\mathbb{C}).$ Let
$\left(  \mathcal{M}_{k}(\mathbb{C})\right)  ^{n}$ denote the direct sum of
$n$ copies of $\mathcal{M}_{k}(\mathbb{C}).$ Let $\mathcal{M}_{k}%
^{s.a}(\mathbb{C})$ be the subalgebra of $\mathcal{M}_{k}(\mathbb{C})$
consisting of all self-adjoint matrices of $\mathcal{M}_{k}(\mathbb{C}).$ Let
$\left(  \mathcal{M}_{k}^{s.a}(\mathbb{C})\right)  ^{n}$ be the direct sum (
or orthogonal sum) of $n$ copies of $\mathcal{M}_{k}^{s.a}(\mathbb{C}).$ Let
$\left\Vert \cdot\right\Vert $ be an operator norm on $\mathcal{M}_{k}\left(
\mathbb{C}\right)  ^{n}$ defined by
\[
\left\Vert \left(  A_{1},\ldots,A_{n}\right)  \right\Vert =\max\left\{
\left\Vert A_{1}\right\Vert ,\ldots,\left\Vert A_{n}\right\Vert \right\}
\]
for all $\left(  A_{1},\ldots,A_{n}\right)  $ in $\mathcal{M}_{k}\left(
\mathbb{C}\right)  ^{n}.$ Let $\left\Vert \cdot\right\Vert _{2}$ denote the
trace norm induced by $\tau_{k}$ on $M_{k}\left(  \mathbb{C}\right)  ^{n}%
,$i.e.,%
\[
\left\Vert \left(  A_{1},\ldots,A_{n}\right)  \right\Vert _{2}=\sqrt{\tau
_{k}(A_{1}^{\ast}A_{1})+\ldots+\tau_{k}(A_{n}^{\ast}A_{n})}%
\]
for all $\left(  A_{1},\ldots,A_{n}\right)  $ in $\mathcal{M}_{k}\left(
\mathbb{C}\right)  ^{n}.$

For every $\omega>0,$ we define the $\omega$-$\left\Vert \cdot\right\Vert
$-ball $Ball(B_{1},\ldots,B_{n};\omega,\left\Vert \cdot\right\Vert )$ centered
at $(B_{1},\ldots,B_{n})$ in $\mathcal{M}_{k}\left(  \mathbb{C}\right)  ^{n}$
to the subset of $\mathcal{M}_{k}\left(  \mathbb{C}\right)  ^{n}$ consisting
of all $(A_{1},\ldots,A_{n})$ in $\mathcal{M}_{k}\left(  \mathbb{C}\right)
^{n}$ such that
\[
\left\Vert \left(  A_{1},\ldots,A_{n}\right)  -(B_{1},\ldots,B_{n})\right\Vert
<\omega.
\]

\begin{definition}
Suppose that $\Sigma$ is a subset of $\mathcal{M}_{k}\left(  \mathbb{C}%
\right)  ^{n}.$ We define $\nu_{\mathcal{1}}(\Sigma,$ $\omega)$ to be the
minimal number of $\omega$-$\left\Vert \cdot\right\Vert $-balls that consist a
covering of $\Sigma$ in $\mathcal{M}_{k}\left(  \mathbb{C}\right)  ^{n}.$
\end{definition}

For every $\omega>0,$ we define the $\omega$-$\left\Vert \cdot\right\Vert
_{2}$-ball $Ball$ $(B_{1},\ldots,B_{n};\omega,\left\Vert \cdot\right\Vert
_{2})$ centered at $(B_{1},\ldots,B_{n})$ in $\mathcal{M}_{k}\left(
\mathbb{C}\right)  ^{n}$ to be the subset of $\mathcal{M}_{k}\left(
\mathbb{C}\right)  ^{n}$ consisting of all $(A_{1},\ldots,A_{n})$ in
$\mathcal{M}_{k}\left(  \mathbb{C}\right)  ^{n}$ such that
\[
\left\Vert (A_{1},\ldots,A_{n})-\left(  B_{1},\ldots,B_{n}\right)  \right\Vert
_{2}<\omega.
\]

\begin{definition}
Suppose that $\Sigma$ is a subset of $\mathcal{M}_{k}\left(  \mathbb{C}%
\right)  ^{n}.$ We define $\nu_{2}(\Sigma,\omega)$ to be the minimal number of
$\omega$-$\left\Vert \cdot\right\Vert _{2}$-balls that consist a covering of
$\Sigma$ in $\mathcal{M}_{k}\left(  \mathbb{C}\right)  ^{n}.$
\end{definition}

\subsection{Unitary orbits of balls in $\mathcal{M}_{k}(\mathbb{C})^{n}$}

For every $\omega>0,$ we define the $\omega$-$orbit$-$\left\Vert
\cdot\right\Vert $-ball $\mathcal{U}(B_{1},\ldots,B_{n};\omega)$ centered at
$(B_{1},\ldots,B_{n})$ in $\mathcal{M}_{k}(\mathbb{C})^{n}$ to be the subset
of $\mathcal{M}_{k}(\mathbb{C})^{n}$ consisting of all $(A_{1},\ldots,A_{n})$
in $\mathcal{M}_{k}(\mathbb{C})^{n}$ such that there exists some unitary
matrix $W$ in $\mathcal{U}(k)$ satisfying
\[
\left\Vert (A_{1},\ldots,A_{n})-(WB_{1}W^{\ast},\ldots,WB_{n}W^{\ast
}\right\Vert <\omega.
\]

\begin{definition}
Suppose that $\Sigma$ is a subset of $\mathcal{M}_{k}(\mathbb{C})^{n}.$ We
define $o_{\mathcal{1}}(\Sigma,\omega)$ to be the minimal number of $\omega
$-$orbit$-$\left\Vert \cdot\right\Vert $-balls that consist a covering of
$\Sigma$ in $\mathcal{M}_{k}(\mathbb{C})^{n}.$
\end{definition}

For every $\omega>0,$ we define the $\omega$-$orbit$-$\left\Vert
\cdot\right\Vert _{2}$-ball $\mathcal{U}(B_{1},\ldots B_{n};\omega,\left\Vert
\cdot\right\Vert _{2})$ centered at $(B_{1},\ldots,B_{n})$ in $\mathcal{M}%
_{k}(\mathbb{C})^{n}$ to be the subset of $\mathcal{M}_{k}(\mathbb{C})^{n}$
consisting of all $(A_{1},\ldots,A_{n})$ in $M_{k}(\mathbb{C})^{n}$ such that
there exists some unitary matrix $W$ in $\mathcal{U}(k)$ satisfying%
\[
\left\Vert (A_{1},\ldots,A_{n})-(WB_{1}W^{\ast},\ldots,WB_{n}W^{\ast
}\right\Vert _{2}<\omega.
\]

\begin{definition}
Suppose that $\Sigma$ is a subset of $\mathcal{M}_{k}(\mathbb{C})^{n}.$ We
define $o_{2}(\Sigma,\omega)$ to be the minimal number of $\omega$%
-$orbit$-$\left\Vert \cdot\right\Vert _{2}-$balls that consist a covering of
$\Sigma$ in $\mathcal{M}_{k}(\mathbb{C})^{n}.$
\end{definition}

\subsection{\bigskip Noncommutative Polynomials}

In this article, we always assume that $\mathcal{A}$ is a unital C*-algebra.
Let $x_{1},\cdots,x_{n},y_{1},\cdots,y_{m}$ be self-adjoint elements in
$\mathcal{A}.$ Let $\mathbb{C}\langle X_{1},\cdots,X_{n},Y_{1},\cdots
,Y_{m}\rangle$ be the set of all noncommutative polynomials in the
indeterminates $X_{1},\cdots,X_{n},Y_{1},\cdots,Y_{m}.$ Let $\mathbb{C}%
_{\mathbb{Q}}=\mathbb{Q+}i\mathbb{Q}$ denote the complex-rational numbers,
i.e., the number whose real and imaginary parts are rational. The set
$\mathbb{C}_{\mathbb{Q}}\langle X_{1},\cdots,X_{n},Y_{1},\cdots,Y_{m}\rangle$
of noncommutative polynomials with complex-rational coefficients is countable.
Throughout this paper we write
\[
\mathbb{C}_{\mathbb{Q}}\langle X_{1},\cdots,X_{n},Y_{1},\cdots,Y_{m}%
\rangle=\left\{  P_{r}:r\in\mathbb{N}\right\}  \text{ and }\mathbb{C}%
_{\mathbb{Q}}\langle X_{1},\cdots,X_{n}\rangle=\left\{  Q_{r}:r\in
\mathbb{N}\right\}
\]
and%
\[
\mathbb{C}_{\mathbb{Q}}\mathbb{\langle}X_{1},X_{2}\cdots\mathbb{\rangle=\cup
}_{m=1}^{\infty}\mathbb{C}_{\mathbb{Q}}\mathbb{\langle}X_{1},\cdots
,X_{m}\mathbb{\rangle}.
\]

\begin{remark}
We always assume that $1\in\mathbb{C}\langle X_{1},\cdots,X_{n},Y_{1}%
,\cdots,Y_{m}\rangle.$
\end{remark}

\subsection{Voiculescu's Norm-microstates Space}

For all integers $r,k\geq1,$ real numbers $R,\varepsilon>0$ and noncommutative
polynomials $P_{1},\ldots,P_{r},$ we define
\[
\Gamma_{R}^{\text{(top)}}(x_{1},\ldots,x_{n},y_{1},\ldots,y_{m};k,\varepsilon
,P_{1},\ldots,P_{r}%
\]
to be the subset of $\left(  \mathcal{M}_{k}^{s.a}(\mathbb{C)}\right)  ^{n+m}$
consisting of all these
\[
\left(  A_{1},\ldots,A_{n},B_{1},\ldots,B_{m}\right)  \in\left(
\mathcal{M}_{k}^{s.a}(\mathbb{C)}\right)  ^{n+m}%
\]
satisfying
\[
\max\left\{  \left\Vert A_{1}\right\Vert ,\ldots,\left\Vert A_{n}\right\Vert
,\left\Vert B_{1}\right\Vert ,\ldots\left\Vert B_{m}\right\Vert \right\}  \leq
R
\]
and%
\[
\left\vert \left\Vert P_{j}(A_{1},\ldots,A_{n},B_{1},\ldots,B_{m})\right\Vert
-\left\Vert P_{j}(x_{1},\ldots,x_{n},y_{1},\ldots,y_{m}\right\Vert \right\vert
\leq\varepsilon,\mathcal{8}1\leq j\leq r.
\]

\begin{remark}
\label{junhao}In the original definition of norm-microstates space in
\cite{DV}, the parameter $R$ was not introduced. Note the following
observation: Let
\[
R>\max\left\{  \left\Vert x_{1}\right\Vert ,\left\Vert x_{2}\right\Vert
,\ldots,\left\Vert x_{n}\right\Vert ,\left\Vert y_{1}\right\Vert
,\ldots,\left\Vert y_{m}\right\Vert \right\}  .
\]
When $r$ is large enough and $\varepsilon$ is small enough,%
\[
\Gamma_{R}^{\text{(top)}}(x_{1},\ldots,x_{n},y_{1},\ldots,y_{m};k,\varepsilon
,P_{1},\ldots,P_{r})=\Gamma_{\mathcal{1}}^{\text{(top)}}(x_{1},\ldots
,x_{n},y_{1},\ldots,y_{m};k,\varepsilon,P_{1},\ldots,P_{r})
\]
for all $k\geq1.$ Our definition agrees with the one in \cite{DV} for large
$R,r$ and small $\varepsilon.$
\end{remark}

Define the norm-microstates space of $x_{1},\ldots,x_{n}$ in the presence of
$y_{1},\ldots,y_{m},$ denoted by
\[
\Gamma_{R}^{\text{(top)}}(x_{1},\ldots,x_{n}:y_{1},\ldots,y_{m};k,\varepsilon
,P_{1},\ldots,P_{r})
\]
as the projection of $\Gamma_{R}^{\text{(top)}}(x_{1},\ldots,x_{n}%
,y_{1},\ldots,y_{m};k,\varepsilon,P_{1},\ldots,P_{r})$ onto the space
$(M_{k}^{s.a}(\mathbb{C}))^{n}$ via the mapping
\[
(A_{1},\ldots,A_{n},B_{1},\ldots,B_{m})\rightarrow(A_{1},\ldots,A_{n}).
\]

\subsection{Voiculescu's topological free entropy dimension}

Define
\[
\nu_{\mathcal{1}}(\Gamma_{R}^{\text{(top)}}(x_{1},\ldots,x_{n}:y_{1}%
,\ldots,y_{m};k,\varepsilon,P_{1},\ldots,P_{r}),\omega)
\]
to be the covering number of the set $\Gamma_{R}^{\text{(top)}}(x_{1}%
,\ldots,x_{n}:y_{1},\ldots,y_{m};k,\varepsilon,P_{1},\ldots,P_{r}) $ by
$\omega$-$\left\Vert \cdot\right\Vert $-balls in the metric space
$(M_{k}^{s.a}(\mathbb{C))}^{n}$ equipped with operator norm.

\begin{definition}
Define%
\begin{align*}
&  \delta_{\text{top}}(x_{1},\ldots,x_{n};\omega)\\
&  =\underset{R>0}{\sup}\underset{\varepsilon>0,r\in\mathbb{N}}{\inf}%
\underset{k\rightarrow\mathcal{1}}{\lim\sup}\frac{\log(\nu_{\mathcal{1}%
}(\Gamma_{R}^{\text{(top)}}(x_{1},\ldots,x_{n};k,\varepsilon,Q_{1}%
,\ldots,Q_{r}),\omega))}{-k^{2}\log\omega}.
\end{align*}

\textbf{The topological free entropy dimension }of $x_{1},\ldots,x_{n}$ is
defined by
\[
\delta_{top}(x_{1},\ldots,x_{n})=\underset{\omega\rightarrow0^{+}}{\lim\sup
}\delta_{top}(x_{1},\ldots,x_{n};\omega).
\]
Similarly, define%
\begin{align*}
\delta_{\text{top}}(x_{1},\ldots,x_{n}  &  :y_{1},\ldots,y_{m};\omega)\\
&  =\underset{R>0}{\sup}\underset{\varepsilon>0,r\in\mathbb{N}}{\inf}%
\underset{k\rightarrow\mathcal{1}}{\lim\sup}\frac{\log(\nu_{\mathcal{1}%
}(\Gamma_{R}^{\text{(top)}}(x_{1},\ldots,x_{n}:y_{1},\ldots,y_{m}%
;k,\varepsilon,P_{1},\ldots,P_{r}),\omega))}{-k^{2}\log\omega}.
\end{align*}
The topological free entropy dimension of $x_{1},\cdots,x_{n}$ in the presence
of $y_{1},\cdots,y_{m}$ is defined by
\[
\delta_{top}(x_{1},\ldots,x_{n}:y_{1},\cdots,y_{m})=\underset{\omega
\rightarrow0^{+}}{\lim\sup}\delta_{top}(x_{1},\ldots,x_{n}:y_{1},\cdots
,y_{m};\omega).
\]

\end{definition}

\begin{remark}
Let $R>\max\left\{  \left\Vert x_{1}\right\Vert ,\ldots,\left\Vert
x_{n}\right\Vert ,\left\Vert y_{1}\right\Vert ,\ldots,\left\Vert
y_{m}\right\Vert \right\}  $ be some positive number. By Remark \ref{junhao},
we know the supremum over $R>0$ is unnecessary, i.e.,
\begin{align*}
\delta_{top}(x_{1},\ldots,x_{n}  &  :y_{1},\ldots,y_{m})\\
&  =\underset{\omega\rightarrow0^{+}}{\lim\sup}\underset{\varepsilon
>0,r\in\mathbb{N}}{\inf}\underset{k\rightarrow\mathcal{1}}{\lim\sup}\frac
{\log(\nu_{\mathcal{1}}(\Gamma_{R}^{\text{(top)}}(x_{1},\ldots,x_{n}%
:y_{1},\ldots,y_{m};k,\varepsilon,P_{1},\ldots,P_{r}),\omega))}{-k^{2}%
\log\omega}%
\end{align*}

\end{remark}

\subsection{$\delta_{\text{top}}^{1/2}$}

In this subsection we recall the definition of $\delta_{\text{top}}^{1/2}$ and
its properties.

\begin{definition}
(\cite{DV}) The norm-semi-microstates $\Gamma_{1/2}^{\text{top}}\left(
x_{1},\ldots,x_{n};k,\varepsilon,Q_{1},\ldots,Q_{r}\right)  $ is the set of
all $\left(  a_{1},\ldots,a_{n}\right)  \in\mathcal{M}_{k}^{n}\left(
\mathbb{C}\right)  $ such that%
\[
\left\Vert Q_{j}\left(  a_{1},\ldots,a_{n}\right)  \right\Vert \leq\left\Vert
Q_{j}\left(  x_{1},\ldots,x_{n}\right)  \right\Vert +\varepsilon
\]
for $1\leq j\leq r$.

We define $\delta_{\text{top}}^{1/2}\left(  x_{1},\ldots,x_{n}\right)  $ to
be
\[
\limsup\limits_{\omega\rightarrow0^{+}}\inf\limits_{r\in\mathbb{N}%
,\varepsilon>0}\limsup\limits_{k\rightarrow\infty}\frac{\log\left(
\nu_{\mathcal{1}}\left(  \Gamma_{1/2}^{\text{top}}\left(  \left(  x_{1}%
,\ldots,x_{n};k,\varepsilon,Q_{1},\ldots,Q_{r}\right)  \right)  ,\omega
\right)  \right)  }{-k^{2}\log\omega}%
\]

\end{definition}

\begin{theorem}
(\cite{HLSW}) $\delta_{\text{top}}^{1/2}\left(  x_{1},\ldots,x_{n}\right)
=\delta_{\text{top}}\left(  x_{1},\ldots,x_{n}\right)  $ whenever
$\delta_{\text{top}}\left(  x_{1},\ldots,x_{n}\right)  $ is defined.
\end{theorem}

\subsection{Topological orbit dimension $\mathfrak{K}_{top}^{(2)}$ and
Modified free entropy dimension $\mathfrak{K}_{2}^{(2)}$}

In this subsection, we are going to recall a C*-algebra invariant "topological
orbit dimension $\mathfrak{K}_{top}^{(2)}$" and its basic properties.

\begin{definition}
(\cite{HLS}) Define%
\[
\mathfrak{K}_{top}^{(2)}(x_{1},\cdots,x_{n};\omega)=\sup_{R>0}\inf
_{\varepsilon>0,r\in\mathbb{N}}\lim\sup_{k\rightarrow\mathcal{1}}\frac
{\log(o_{2}(\Gamma_{R}^{(top)}(x_{1},\cdots,x_{n};k,\varepsilon,Q_{1}%
,\cdots,Q_{r}),\omega)}{k^{2}}%
\]
and%
\[
\mathfrak{K}_{top}^{(2)}(x_{1},\cdots,x_{n})=\sup_{\omega>0}\mathfrak{K}%
_{top}^{(2)}(x_{1},\cdots,x_{n};\omega)=\lim_{\omega\rightarrow0^{+}%
}\mathfrak{K}_{top}^{(2)}(x_{1},\cdots,x_{n};\omega).
\]
Similarly, define%
\begin{align*}
\mathfrak{K}_{top}^{(2)}(x_{1},\cdots,x_{n}  &  :y_{1},\cdots,y_{m};\omega)=\\
&  \sup_{R>0}\inf_{\varepsilon>0,r\in\mathbb{N}}\lim\sup_{k\rightarrow
\mathcal{1}}\frac{\log(o_{2}(\Gamma_{R}^{(top)}(x_{1},\cdots,x_{n}%
:y_{1},\cdots,y_{m};k,\varepsilon,P_{1},\cdots,P_{r}),\omega)}{k^{2}}%
\end{align*}
and
\[
\mathfrak{K}_{top}^{(2)}(x_{1},\cdots,x_{n}:y_{1},\cdots,y_{m})=\sup
_{\omega>0}\mathfrak{K}_{top}^{(2)}(x_{1},\cdots,x_{n};\omega)=\lim
_{\omega\rightarrow0^{+}}\mathfrak{K}_{top}^{(2)}(x_{1},\cdots,x_{n}%
:y_{1},\cdots,y_{m};\omega)
\]

\end{definition}

The topological orbit dimension $\mathfrak{K}_{top}^{(2)}$ is in fact a
C*-algebra invariant. In view of this result, we use $\mathfrak{K}_{top}%
^{(2)}\left(  \mathcal{A}\right)  $ to denote $\mathfrak{K}_{top}^{(2)}\left(
x_{1},\cdots,x_{n}\right)  $ for an arbitrary generating set $\left\{
x_{1},\cdots,x_{n}\right\}  $ for $\mathcal{A}$.

\begin{theorem}
(\cite{HLS}) \label{in}Suppose that $\mathcal{A}$ is a unital C*-algebra and
$\left\{  x_{1},\cdots,x_{n}\right\}  ,\left\{  y_{1},\cdots,y_{p}\right\}  $
are two families of self-adjoint generators of $\mathcal{A}$. Then
\[
\mathfrak{K}_{top}^{(2)}(x_{1},\cdots,x_{n})=\mathfrak{K}_{top}^{(2)}%
(y_{1},\cdots,y_{p})
\]

\end{theorem}

After slightly modify the proof of Theorem \ref{in}, we can conclude that

\begin{theorem}
\label{ind}Suppose that $\mathcal{A}$ is a unital C*-algebra and $x_{1}%
,\cdots,x_{n},y_{1},\cdots,y_{p},w_{1},\cdots,w_{t}$ are self-adjoint elements
in $\mathcal{A}$. If $C^{\ast}\left(  x_{1},\cdots,x_{n}\right)  =C^{\ast
}\left(  y_{1},\cdots,y_{p}\right)  ,$ then
\[
\mathfrak{K}_{top}^{(2)}(x_{1},\cdots,x_{n}:w_{1},\cdots,w_{t})=\mathfrak{K}%
_{top}^{(2)}(y_{1},\cdots,y_{p}:w_{1},\cdots,w_{t})
\]

\end{theorem}

\begin{remark}
\label{indd}From the definition, it is clear that

\begin{enumerate}
\item $\mathfrak{K}_{top}^{(2)}(x_{1},\cdots,x_{n}:y_{1},\cdots,y_{p}%
)\geq\mathfrak{K}_{top}^{(2)}(x_{1},\cdots,x_{n}:y_{1},\cdots y_{p},y_{p+1});$

\item If $\mathfrak{K}_{top}^{(2)}(x_{1},\cdots,x_{n}:x_{1},\cdots,x_{n+j})=0$
$(j\geq0),then$%
\[
\mathfrak{K}_{top}^{(2)}(x_{1},\cdots,x_{n-1}:x_{1},\cdots,x_{n+j})=0
\]

\end{enumerate}
\end{remark}

Let $\mathcal{M}$ be a von Neumann algebra with a tracial state $\tau,$ and
let $x_{1},\cdots,x_{n}$ be self-adjoint elements in $\mathcal{M}$. For any
positive $R$ and $\varepsilon,$ and any $m,k\in\mathbb{N}$, let $\Gamma
_{R}\left(  x_{1},\cdots,x_{n};m,k,\varepsilon;\tau\right)  $ be the subset of
$\mathcal{M}_{k}^{s.a.}\left(  \mathbb{C}\right)  ^{n}$ consisting of all
$\left(  A_{1},\cdots,A_{n}\right)  $ in $\mathcal{M}_{k}^{s.a.}\left(
\mathbb{C}\right)  ^{n}$ such that
\[
\max_{1\leq j\leq n}\left\Vert A_{j}\right\Vert \leq R\text{ and }\left\vert
\tau_{k}\left(  A_{i_{1}}\cdots A_{i_{q}}\right)  -\tau\left(  x_{i_{1}}\cdots
x_{i_{q}}\right)  \right\vert <\varepsilon,
\]
for all $1\leq i_{1},\cdots,i_{q}\leq n$ and $1\leq q\leq m.$

For any $\omega>0,$ let $o_{2}\left(  \Gamma_{R}\left(  x_{1},\cdots
,x_{n};m,k,\varepsilon;\tau\right)  ,\omega\right)  $ be the minimal number of
$\omega$-orbit-$\left\Vert \cdot\right\Vert _{2}$-balls in $\mathcal{M}%
_{k}\left(  \mathbb{C}\right)  ^{n}$ that constitute a covering of $\Gamma
_{R}\left(  x_{1},\cdots,x_{n};m,k,\varepsilon;\tau\right)  .$ Now we define,
successively,%
\[
\mathfrak{K}_{2}^{\left(  2\right)  }\left(  x_{1},\cdots,x_{n};\omega
;\tau\right)  =\sup_{R>0}\inf_{m\in\mathbb{N},\varepsilon>0}\lim
\sup_{k\rightarrow\mathcal{1}}\frac{\log\left(  o_{2}\left(  \Gamma_{R}\left(
x_{1},\cdots,x_{n};m,k,\varepsilon;\tau\right)  \right)  \right)  }{k^{2}}%
\]%
\[
\mathfrak{K}_{2}^{\left(  2\right)  }\left(  x_{1},\cdots,x_{n};\tau\right)
=\lim\sup_{\omega\rightarrow0^{+}}\mathfrak{K}_{2}^{\left(  2\right)  }\left(
x_{1},\cdots,x_{n};\omega;\tau\right)
\]
where $\mathfrak{K}_{2}^{\left(  2\right)  }\left(  x_{1},\cdots,x_{n}%
;\tau\right)  $ is called the modified free orbit-dimension of $x_{1}%
,\cdots,x_{n}$ with respect to the tracial state $\tau$ \cite{HLS}.

\begin{remark}
(\cite{HLS}) \label{inddd}Suppose $x_{1},\cdots,x_{n}$ is a family of
self-adjoint elements in a von Neumann algebra with a tracial state $\tau.$
Let $\mathfrak{K}_{2}\left(  x_{1},\cdots,x_{n};\tau\right)  $ be the upper
orbit dimension of $x_{1},\cdots,x_{n}$ defined in Definition 1 of \cite{HS1}.
Then $\mathfrak{K}_{2}^{\left(  2\right)  }\left(  x_{1},\cdots,x_{n}%
;\tau\right)  =0$ if $\mathfrak{K}_{2}\left(  x_{1},\cdots,x_{n};\tau\right)
=0.$
\end{remark}

\subsection{MF-Traces and MF Nuclear Algebras}

We note that the definition of $\delta_{top}\left(  x_{1},\cdots,x_{n}\right)
$ makes sense if and only if, for every $\varepsilon>0$ and every $r,k_{0}%
\in\mathbb{N}$, there is a $k\geq k_{o}$ such that
\[
\Gamma^{(top)}\left(  x_{1},\cdots,x_{n};k,\varepsilon,Q_{1},\cdots
,Q_{r}\right)  \neq\varnothing.
\]
In \cite{HS2}, it has shown that this is equivalent to $C^{\ast}\left(
x_{1},\cdots,x_{n}\right)  $ being an MF C*-algebra in the sense of Blackadar
and Kirchberg \cite{BK1}. A C*-algebra $\mathcal{A}$ is an MF-algebra if
$\mathcal{A}$ can be embedded into $\Pi_{1\leq k<\mathcal{1}}\mathcal{M}%
_{m_{k}}\left(  \mathbb{C}\right)  /\sum_{1\leq k<\mathcal{1}}\mathcal{M}%
_{m_{k}}\left(  \mathbb{C}\right)  $ for some increasing sequence $\left\{
m_{k}\right\}  $ of positive integers. In particular C*$\left(  x_{1}%
,\cdots,x_{n}\right)  $ is an MF-algebra if there is a sequence $\left\{
m_{k}\right\}  $ of positive integers and sequences $\left\{  A_{1k}\right\}
,\cdots,\left\{  A_{nk}\right\}  $ with $A_{1k},\cdots,A_{nk}\in
\mathcal{M}_{m_{k}}\left(  \mathbb{C}\right)  $ such that%
\[
\lim_{k\rightarrow\mathcal{1}}\left\Vert Q\left(  A_{1k},\cdots,A_{nk}\right)
\right\Vert =\left\Vert Q\left(  x_{1},\cdots,x_{n}\right)  \right\Vert
\]
for every *-polynomial $Q\left(  t_{1},\cdots,t_{n}\right)  .$

\begin{definition}
(\cite{HLSW}) Suppose $\mathcal{A=}C^{\ast}\left(  x_{1},\cdots,x_{n}\right)
$ is an MF C*-algebra A tracial state $\tau$ on $\mathcal{A}$ is an MF-trace
if there is sequence $\left\{  m_{k}\right\}  $ of positive integers and
sequences $\left\{  A_{1k}\right\}  ,\cdots,\left\{  A_{nk}\right\}  $ with
$A_{1k},\cdots,A_{nk}\in\mathcal{M}_{m_{k}}\left(  \mathbb{C}\right)  $ such
that, for every $\ast$-polynomial $p,$

\begin{enumerate}
\item $\lim_{k\rightarrow\mathcal{1}}\left\Vert Q\left(  A_{1k},\cdots
,A_{nk}\right)  \right\Vert =\left\Vert Q\left(  x_{1},\cdots,x_{n}\right)
\right\Vert ,$ and

\item $\lim_{k\rightarrow\mathcal{1}}\tau_{m_{k}}\left(  Q\left(
A_{1k},\cdots,A_{nk}\right)  \right)  =\tau\left(  Q\left(  x_{1},\cdots
,x_{n}\right)  \right)  .$
\end{enumerate}
\end{definition}

We let $\mathcal{TS}\left(  \mathcal{A}\right)  $ denote the set of all
tracial states on $\mathcal{A}$ and $\mathcal{T}_{MF}\left(  \mathcal{A}%
\right)  $ denote the set of all MF-traces on $\mathcal{A}$.

\begin{definition}
(\cite{HLSW}) A C*-algebra $\mathcal{A=}C^{\ast}\left(  x_{1},\cdots
,x_{n}\right)  $ is MF-nuclear if $\pi_{\tau}\left(  \mathcal{A}\right)
^{\prime\prime}$ is hyperfinite for every $\tau\in\mathcal{T}_{MF}\left(
\mathcal{A}\right)  $ where $\pi_{\tau}$ is the GNS representation of
$\mathcal{A}$ with respect to $\tau.$
\end{definition}

\section{Relation between $\mathfrak{K}_{top}^{(2)}$ and $\mathfrak{K}%
_{2}^{(2)}$}

\begin{definition}
(\cite{HLS}) Suppose that $\mathcal{A}$ is a unital C*-algebra and
$\mathcal{TS}\left(  \mathcal{A}\right)  $ is the set of all tracial states of
$\mathcal{A}$. Suppose that $x_{1},\cdots,x_{n}$ is a family of self-adjoint
elements in $\mathcal{A}$. Define
\[
\mathfrak{KK}_{2}^{\left(  2\right)  }\left(  x_{1},\cdots,x_{n}\right)
=\sup_{\tau\in\mathcal{T}S\left(  \mathcal{A}\right)  }\mathfrak{K}%
_{2}^{\left(  2\right)  }\left(  x_{1},\cdots,x_{n};\tau\right)
\]

\end{definition}

\begin{theorem}
(\cite{HLS}) \ Suppose that $\mathcal{A}$ is a unital C*-algebra and
$x_{1},\cdots,x_{n}$ is a family of self-adjoint generating elements in
$\mathcal{A}$. Then
\[
\mathfrak{K}_{top}^{(2)}(x_{1},\cdots,x_{n})\leq\mathfrak{KK}_{2}^{\left(
2\right)  }\left(  x_{1},\cdots,x_{n}\right)
\]

\end{theorem}

We can generalize the preceding theorem as follows. The proof is similar to
the proof in \cite{HLS}, but for completeness, we give its proof here.

\begin{theorem}
\label{li}Let $\mathcal{A}$ be a unital C*-algebra and $\{x_{1},\cdots
,x_{n},y_{1},\cdots,y_{m}\}$ be a family of self-adjoint generating elements
in $\mathcal{A}$. Then
\[
\mathfrak{K}_{top}^{(2)}(x_{1},\cdots,x_{n}:y_{1},\cdots,y_{m})\leq\sup
_{\tau\in\mathcal{T}_{MF}\left(  \mathcal{A}\right)  }\mathfrak{K}%
_{2}^{\left(  2\right)  }\left(  x_{1},\cdots,x_{n}:y_{1},\cdots,y_{m}%
;\tau\right)
\]
where $\mathcal{T}_{MF}\left(  \mathcal{A}\right)  $ is the set of all
MF-tracial states on $\mathcal{A}$.
\end{theorem}

\begin{proof}
If $\mathfrak{K}_{top}^{(2)}(x_{1},\cdots,x_{n}:y_{1},\cdots,y_{m})=0,$ the
inequality holds automatically. Assume
\[
\mathfrak{K}_{top}^{(2)}(x_{1},\cdots,x_{n}:y_{1},\cdots,y_{m})>\alpha>0,
\]
we need to show that $\mathfrak{K}_{2}^{\left(  2\right)  }\left(
x_{1},\cdots,x_{n}:y_{1},\cdots,y_{m};\tau\right)  >\alpha>0$ for some MF
trace $\tau.$ Let $\left\{  P_{r}\right\}  _{r=1}^{\mathcal{1}}$ be a family
of noncommutative polynomials in $\mathbb{C}\left\langle X_{1},\cdots
,X_{n},Y_{1},\cdots,Y_{m}\right\rangle $ with rational coefficients. Let
$R>\max\left\{  \left\Vert x_{1}\right\Vert ,\cdots,\left\Vert x_{n}%
\right\Vert ,\left\Vert y_{1}\right\Vert ,\cdots,\left\Vert y_{m}\right\Vert
\right\}  .$ From the assumption that $\mathfrak{K}_{top}^{(2)}(x_{1}%
,\cdots,x_{n}:y_{1},\cdots,y_{m})>\alpha,$ it follows that there exist a
positive number $\omega_{0}>0$ and a sequence of positive integers $\left\{
k_{q}\right\}  _{q=1}^{\mathcal{1}}$ with $k_{1}<k_{2}<\cdots,$ so that for
some $\alpha^{\prime}>\alpha$%
\[
\lim_{q\longrightarrow\mathcal{1}}\frac{\log\left(  o_{2}\left(  \Gamma
_{R}^{\left(  top\right)  }\left(  x_{1},\cdots,x_{n}:y_{1},\cdots,y_{m}%
;k_{q},\frac{1}{q},P_{1},\cdots,P_{q}\right)  ,\omega_{0}\right)  \right)
}{k_{q}^{2}}>\alpha^{\prime}.
\]
Let $\mathcal{A}\left(  n+m\right)  $ be the universal unital C*-algebra
generated by self-adjoint elements
\[
a_{1},\cdots,a_{n},a_{n+1},\cdots,a_{n+m}%
\]
of norm $R$, that is the unital full free product of $n+m$ copies of $C\left[
-R,R\right]  .$ A microstate
\begin{align*}
\eta &  =\left(  A_{1},\cdots,A_{n},B_{1},\cdots,B_{m}\right) \\
&  \in\Gamma_{R}^{\left(  top\right)  }\left(  x_{1},\cdots,x_{n},y_{1}%
,\cdots,y_{m};k_{q},\frac{1}{q},P_{1},\cdots,P_{q}\right)  =\Gamma\left(
q\right)
\end{align*}
define a unital *-homomorphism $\varphi_{\eta}:\mathcal{A}\left(  n+m\right)
\longrightarrow\mathcal{M}_{k_{q}}\left(  \mathbb{C}\right)  $ so that
$\varphi_{\eta}\left(  a_{i}\right)  =A_{i}\left(  1\leq i\leq n\right)  $ and
$\varphi_{\eta}\left(  a_{i}\right)  =B_{j}\left(  n+1\leq i\leq n+m\right)  $
as well as a tracial state $\tau_{\eta}\in\mathcal{TS}\left(  \mathcal{A}%
\left(  n+m\right)  \right)  $ with $\tau_{\eta}=\frac{Tr_{k_{q}}\circ
\varphi_{\eta}}{k_{q}}.$ Similarly there is a *-homomorphism $\varphi
:\mathcal{A}\left(  n+m\right)  \longrightarrow\mathcal{A}$ so that
\[
\varphi\left(  a_{i}\right)  =x_{i}\text{ \ \ \ }\left(  1\leq i\leq n\right)
,
\]%
\[
\text{and }\varphi\left(  b_{j}\right)  =y_{j}\text{ }\left(  n+1\leq i\leq
n+m\right)  .
\]

It is not hard to see that the weak topology on $\Omega=\mathcal{TS}\left(
\mathcal{A}\left(  n+m\right)  \right)  $ is induced by the metric
\[
d\left(  \tau_{1},\tau_{2}\right)  =\sum_{s=1}^{\mathcal{1}}\sum_{i_{1}\cdots
i_{s}\in\left(  \left\{  1,\cdots,n+m\right\}  \right)  ^{s}}\left(  2R\left(
n+m\right)  \right)  ^{-s}\left\vert \left(  \tau_{1}-\tau_{2}\right)  \left(
t_{i_{1}}\cdots t_{i_{s}}\right)  \right\vert \text{ }%
\]
where $t_{i}\in\left\{  a_{1},\cdots,a_{n+m}\right\}  .$ Therefore $\Omega$ is
a compact metric space and
\[
K_{q}=\left\{  \tau_{\eta}\in\Omega|\eta\in\Gamma\left(  q\right)  \right\}
\]
is a compact subset of $\Omega$ because $\eta\longrightarrow\tau_{\eta}$ is
continuous and $\Gamma\left(  q\right)  $ is compact. Let further
$K\subseteq\Omega$ denote the compact subset $\left(  \mathcal{TS}\left(
\mathcal{A}\right)  \right)  \circ\varphi.$ Given $\varepsilon>0,$ form the
fact that $\Omega$ is compact it follows that there is some $L\left(
\varepsilon\right)  >0$ so that for each $q\geq1,$%
\[
K_{q}=K_{q}^{1}\cup\cdots\cup K_{q}^{L\left(  \varepsilon\right)  }%
\]
where each compact set $K_{q}^{j}$ has diameter$<\varepsilon.$ Define
\[
\Phi:\underset{n+m}{\underbrace{\mathcal{M}_{k_{q}}\left(  \mathbb{C}\right)
\oplus\cdots\oplus\mathcal{M}_{k_{q}}\left(  \mathbb{C}\right)  }%
}\longrightarrow\underset{n}{\underbrace{\mathcal{M}_{k_{q}}\left(
\mathbb{C}\right)  \oplus\cdots\oplus\mathcal{M}_{k_{q}}\left(  \mathbb{C}%
\right)  }}%
\]
by $\Phi\left(  A_{1},\cdots,A_{n},B_{1},\cdots,B_{m}\right)  =\left(
A_{1},\cdots,A_{n}\right)  .$ Let
\[
\Gamma\left(  q,j\right)  =\left\{  \eta\in\Gamma\left(  q\right)  |\tau
_{\eta}\in K_{q}^{j}\right\}
\]
We have $\Gamma\left(  q\right)  =\Gamma\left(  q,1\right)  \cup\cdots
\cup\Gamma\left(  q,L\left(  \varepsilon\right)  \right)  .$ Let further
$\Gamma^{\prime}\left(  q\right)  $ denote some $\Gamma\left(  q,j\right)  $
such that%
\[
o_{2}\left(  \Phi\left(  \Gamma^{\prime}\left(  q\right)  \right)  ,\omega
_{0}\right)  \geq\frac{o_{2}\left(  \Phi\left(  \Gamma\left(  q\right)
\right)  ,\omega_{0}\right)  }{L\left(  \varepsilon\right)  }.
\]
Thus we have $\lim_{q\longrightarrow\mathcal{1}}\frac{\log\left(  o_{2}\left(
\Phi\left(  \Gamma^{\prime}\left(  q\right)  \right)  ,\omega_{0}\right)
\right)  }{k_{q}^{2}}>\alpha^{\prime}.$ Given $\varepsilon$ successively the
values $1,1/2,\cdots,1/s,\cdots,$ we can find a subsequence $\left\{
q_{s}\right\}  _{s=1}^{\mathcal{1}}$ such that the chosen set $K_{q_{s}%
}^{j_{s}}\subseteq K_{q_{s}}$ has diameter $<\frac{1}{\varepsilon}$ and the
corresponding set $\Gamma^{\prime}\left(  q_{s}\right)  $ satisfying
\[
\lim_{q\longrightarrow\mathcal{1}}\frac{\log\left(  o_{2}\left(  \Phi\left(
\Gamma^{\prime}\left(  q_{s}\right)  \right)  ,\omega_{0}\right)  \right)
}{k_{q_{s}}^{2}}>\alpha^{\prime}%
\]
Without loss of generality, we can assume that $\tau$ is the weak limit of
some sequence $\left(  \tau_{\eta\left(  q_{s}\right)  }\right)
_{s=1}^{\mathcal{1}}.$ Then $\tau\in K.$ In fact
\[
\left\vert \tau\left(  Q\left(  a_{1},\cdots,a_{n},b_{1},\cdots,b_{m}\right)
\right)  \right\vert =\lim_{s\longrightarrow\mathcal{1}}\left\vert \tau
_{\eta\left(  q_{s}\right)  }\left(  Q\left(  a_{1},\cdots,a_{n},b_{1}%
,\cdots,b_{m}\right)  \right)  \right\vert ,
\]
therefore $\tau$ is an MF trace.

We can further assume that there is a subsequence $\left\{  q_{s\left(
t\right)  }\right\}  _{t=1}^{\mathcal{1}}$ of $\left\{  q_{s}\right\}
_{s=1}^{\mathcal{1}}$ so that the chosen set $K_{q_{s}\left(  t\right)
}^{j_{s}\left(  t\right)  }\subseteq K_{q_{s}\left(  t\right)  }\subseteq
B\left(  \tau,\frac{1}{t}\right)  ,$ the ball of radius 1/t and center $\tau.$
Therefore, for any $m\in\mathbb{N}$ and $\varepsilon>0$, we have%
\[
\Gamma^{\prime}\left(  q_{s\left(  t\right)  }\right)  \subseteq\Gamma
_{R}\left(  x_{1},\cdots,x_{n},y_{1},\cdots,y_{m};k_{q_{s}\left(  t\right)
},m,\varepsilon;\tau\right)
\]
when $t$ is large enough. Thus $\Phi\left(  \Gamma^{\prime}\left(  q_{s\left(
t\right)  }\right)  \right)  \subseteq\Phi\left(  \Gamma_{R}\left(
x_{1},\cdots,x_{n},y_{1},\cdots,y_{m};k_{q_{s}\left(  t\right)  }%
,m,\varepsilon;\tau\right)  \right)  .$ Hence
\begin{align*}
\mathfrak{K}_{2}^{\left(  2\right)  }\left(  x_{1},\cdots,x_{n}:y_{1}%
,\cdots,y_{m};\tau\right)   &  \geq\mathfrak{K}_{2}^{\left(  2\right)
}\left(  x_{1},\cdots,x_{n}:y_{1},\cdots,y_{m};\omega_{0};\tau\right) \\
&  \geq\lim_{t\longrightarrow\mathcal{1}}\frac{\log o_{2}\left(  \Phi\left(
\Gamma^{\prime}\left(  q_{s\left(  t\right)  }\right)  \right)  ,\omega
_{0}\right)  }{k_{q_{s\left(  t\right)  }}^{2}}>\alpha^{\prime},
\end{align*}
and hence
\[
\mathfrak{K}_{top}^{(2)}(x_{1},\cdots,x_{n}:y_{1},\cdots,y_{m})\leq\sup
_{\tau\in\mathcal{T}_{MF}\left(  \mathcal{A}\right)  }\mathfrak{K}%
_{2}^{\left(  2\right)  }\left(  x_{1},\cdots,x_{n}:y_{1},\cdots,y_{m}%
;\tau\right)  .
\]

\end{proof}

Now we are ready to simplify the proof of the following theorem.

\begin{theorem}
\label{don}(\cite{HLSW}) Suppose $\mathcal{A}$ is an MF-nuclear C*-algebra
with a family of self-adjoint generators $x_{1},\cdots,x_{n}.$ Then
\[
\delta_{top}\left(  x_{1},\cdots,x_{n}\right)  \leq1.
\]

\end{theorem}

\begin{proof}
It is known that the GNS representation of an MF nuclear C*-algebra with
respect to an MF tracial state yields an injective von Neumann algebra. From
\cite{HS1} and Remark \ref{inddd}
\[
\mathfrak{K}_{2}^{\left(  2\right)  }\left(  x_{1},\cdots,x_{n};\tau\right)
=0\text{ for any }\tau\in\mathcal{TS}_{MF}\left(  \mathcal{A}\right)
\]
where $\mathcal{TS}_{MF}\left(  \mathcal{A}\right)  $ is the set of all MF
tracial states of $\mathcal{A}$. So, from Theorem \ref{li}, we know that
$\mathfrak{K}_{top}^{\left(  2\right)  }\left(  x_{1},\cdots,x_{n}\right)
=0.$ Hence by Theorem 3.1.2 in \cite{HLS},
\[
\delta_{top}\left(  x_{1},\cdots,x_{n}\right)  \leq1.
\]

\end{proof}

\section{Topological free entropy dimension of inner quasidiagonal
C*-algebras}

In this section, we will first recall the concept of inner quasidiagonal
C*-algebras which was first introduced by Blackadar and Kirchberg in
\cite{BK2}. After that, we are going to analyze the topological free entropy
dimension of an MF-nuclear and inner quasidiagonal C*-algebra.

\begin{definition}
(\cite{BK2}) If $\mathcal{B}$ is a C*-algebra, then a projection
$p\in\mathcal{B}$ is in the socle of $\mathcal{B}$ if $p\mathcal{B}p$ is
finite-dimensional. Denote the set of such projections by $socle\left(
\mathcal{B}\right)  .$
\end{definition}

\begin{proposition}
(\cite{BK2}) A C*-algebra $\mathcal{A}$ is inner quasidiagonal if and only if,
for any $x_{1},\cdots,x_{m}\in\mathcal{A}$ and $\varepsilon>0,$ there is a
projection $p\in socle\left(  \mathcal{A}^{\ast\ast}\right)  $ with
$\left\Vert px_{j}p\right\Vert >\left\Vert x_{j}\right\Vert -\varepsilon$ and
$\left\Vert px_{j}-x_{j}p\right\Vert <\varepsilon$ for all $j.$
\end{proposition}

\begin{theorem}
\label{11b}(\cite{BK3}) A separable C*-algebra is inner QD if and only if it
has a separating family of quasidiagonal irreducible representations.
\end{theorem}

It is well-known that every residually finite-dimensional (RFD) C*-algebra is
inner quasidiagonal. In \cite{HLSW}, it has shown that $\delta_{top}\left(
x_{1},\cdots,x_{n}\right)  =1-\frac{1}{\dim\mathcal{A}}$ in which
$\mathcal{A}=C^{\ast}\left(  x_{1},\cdots,x_{n}\right)  $ is MF-nuclear and
residually finite-dimensional. Next theorem will generalize this result to
inner quasidiagonal C*-algebras with finite generators.

\begin{theorem}
\label{li2}Suppose $\mathcal{A=}C^{\ast}\left(  x_{1},\cdots,x_{n}\right)  $
is MF-nuclear and inner quasidiagonal, then $\delta_{top}\left(  x_{1}%
,\cdots,x_{n}\right)  =1-\frac{1}{\dim\mathcal{A}}.$
\end{theorem}

\begin{proof}
If $\dim\mathcal{A<1}$, then $\mathcal{A}$ is RFD. So the conclusion is
followed from Corollary 5.4 in \cite{HLSW}. Now suppose $\dim\mathcal{A=1}$.
Since $\mathcal{A}$ is MF-nuclear, we have $\delta_{top}(x_{1},\cdots
,x_{n})\leq1$ by Theorem \ref{don}.

Let $\mathcal{F}_{0}=$ $\left\{  x_{1},\cdots,x_{n}\right\}  .$ Suppose
\[
\left\{  1\right\}  \cup\mathcal{F}_{0}\subseteq\mathcal{F}_{1}\subseteq
\mathcal{F}_{2}\mathcal{\subseteq\cdots}%
\]
be the sequence of finite subsets of $\mathcal{A}$ such that $\overline
{\cup_{i}\mathcal{F}_{i}}=\mathcal{A}.$ So by Theorem \ref{11b} and the
property of quasidiagonal C*-algebras, we can find a increasing sequence of
projections $\left\{  P_{t}\right\}  _{t=0}^{\mathcal{1}}\in socle\left(
\mathcal{A}^{\ast\ast}\right)  $ such that $\left\Vert P_{t}x-xP_{t}%
\right\Vert <\frac{\varepsilon}{2^{t}}$ and $\left\Vert P_{t}xP_{t}\right\Vert
>\left\Vert x\right\Vert -\frac{\varepsilon}{2^{t}}$ as $x\in\mathcal{F}_{t}.$
Note that
\[
\left(  P_{i}-P_{i-1}\right)  \mathcal{A}\left(  P_{i}-P_{i-1}\right)
\subseteq\mathcal{A}^{\ast\ast}.
\]
Since $P_{t}\mathcal{A}^{\ast\ast}P_{t}=P_{t}\mathcal{A}P_{t},$ then
\begin{align*}
\left(  P_{i}-P_{i-1}\right)  \mathcal{A}\left(  P_{i}-P_{i-1}\right)   &
=P_{t}\left(  \left(  P_{i}-P_{i-1}\right)  \mathcal{A}\left(  P_{i}%
-P_{i-1}\right)  \right)  P_{t}\\
&  \subseteq P_{t}\mathcal{A}^{\ast\ast}P_{t}=P_{t}\mathcal{A}P_{t}\text{
}\ \ \ \ \ \text{as }t\geq i.
\end{align*}
Therefore%
\[
P_{0}\mathcal{A}P_{0}\oplus\left(  P_{1}-P_{0}\right)  \mathcal{A}\left(
P_{1}-P_{0}\right)  \oplus\cdots\oplus\left(  P_{t}-P_{t-1}\right)
\mathcal{A}\left(  P_{t}-P_{t-1}\right)  \subseteq P_{t}\mathcal{A}P_{t}%
\]
It follows that
\begin{align*}
&  \dim\left(  P_{t}\mathcal{A}P_{t}\right) \\
&  \geq\dim\left(  P_{0}\mathcal{A}P_{0}\oplus\left(  P_{1}-P_{0}\right)
\mathcal{A}\left(  P_{1}-P_{0}\right)  \oplus\cdots\oplus\left(  P_{t}%
-P_{t-1}\right)  \mathcal{A}\left(  P_{t}-P_{t-1}\right)  \right)
\end{align*}
Then $\dim\left(  P_{t}\mathcal{A}P_{t}\right)  \rightarrow\mathcal{1}$ as
$t\rightarrow\mathcal{1}$.

Note that $\left\Vert P_{t}x_{j}-x_{j}P_{t}\right\Vert \rightarrow0$ and
$\left\Vert P_{t}x_{j}P_{t}\right\Vert \rightarrow\left\Vert x_{j}\right\Vert
$ as $t\rightarrow\mathcal{1}$ for $1\leq j\leq n,$ so for any $0<\varepsilon
_{0}$ and $r_{0}\in\mathbb{N}$, there are $\varepsilon_{1},r_{1},t_{1}>0$ such
that for every $t>t_{1},0<\varepsilon<\varepsilon_{1}$ and $r>r_{1}$
\[
\Gamma_{1/2}^{\text{top}}\left(  P_{t}x_{1}P_{t},\cdots,P_{t}x_{n}%
P_{t};k,\varepsilon,Q_{1},\cdots,Q_{r}\right)  \subseteq\Gamma_{1/2}%
^{\text{top}}(x_{1},\cdots,x_{n};k,\varepsilon_{0},Q_{1},\cdots,Q_{r_{0}})
\]
for every $k\in\mathbb{N}$. Therefore for any $\omega>0$
\begin{align}
&  \sup_{t>t_{1}}\inf_{\varepsilon<\varepsilon_{1},r>r_{1}}\underset
{k\longrightarrow\mathcal{1}}{\lim\sup}\frac{\log\left(  v_{\left\Vert
\cdot\right\Vert }\left(  \Gamma_{1/2}^{\text{top}}\left(  P_{t}x_{1}%
P_{t},\cdots,P_{t}x_{n}P_{t};k,\varepsilon,Q_{1},\cdots,Q_{r}\right)
,\omega\right)  \right)  }{-k^{2}\log\omega}\nonumber\\
&  \leq\lim\sup_{k\longrightarrow\mathcal{1}}\frac{\log\left(  v_{\left\Vert
\cdot\right\Vert }\left(  \Gamma_{1/2}^{(top)}(x_{1},\cdots,x_{n}%
;k,\varepsilon_{0},Q_{1},\cdots,Q_{r_{0}}),\omega\right)  \right)  }%
{-k^{2}\log\omega}.
\end{align}
It is obvious that $C^{\ast}(P_{t}x_{1}P_{t},\cdots,P_{t}x_{n}P_{t})\subseteq
P_{t}\mathcal{A}P_{t}$ by the fact that $P_{t}\in socle\left(  \mathcal{A}%
^{\ast\ast}\right)  .$ Let $\dim\left(  P_{t_{0}}\mathcal{A}P_{t_{0}}\right)
=M_{t_{0}}<\mathcal{1}$. Then we can find $\left\{  a_{1},\cdots,a_{M_{t_{0}}%
}\right\}  \subseteq\mathcal{A}$ such that
\[
\left\{  P_{t_{0}}a_{1}P_{t_{0}},\cdots,P_{t_{0}}a_{M_{t_{0}}}P_{t_{0}%
}\right\}
\]
is a linearly independent family where $P_{t_{0}}a_{i}P_{t_{0}}$ is an element
with norm 1 for every $i=1,\cdots,M_{t_{0}}$. For any $\varepsilon>0,$ there
are polynomials
\[
H_{1}\left(  X_{1},\cdots,X_{n}\right)  ,\cdots,H_{M_{t_{0}}}\left(
X_{1},\cdots,X_{n}\right)
\]
in $X_{1},\cdots,X_{n}$ such that
\[
\left\Vert a_{i}-H_{i}\left(  x_{1},\cdots,x_{n}\right)  \right\Vert
<\varepsilon\text{ for every }1\leq i\leq M_{t_{0}}%
\]
It follows that
\[
\left\Vert P_{t}a_{i}P_{t}-P_{t}H_{i}\left(  x_{1},\cdots,x_{n}\right)
P_{t}\right\Vert <\varepsilon.
\]
Since $\left\Vert P_{t}x_{j}-x_{j}P_{t}\right\Vert \rightarrow0$ as
$t\rightarrow\mathcal{1}$ for every $1\leq j\leq n$, then there is an integer
$L$ such that
\[
\left\Vert P_{t}a_{i}P_{t}-H_{i}\left(  P_{t}x_{1}P_{t},\cdots,P_{t}x_{n}%
P_{t}\right)  \right\Vert <L\cdot\varepsilon\text{ for every }1\leq i\leq
M_{t_{0}}%
\]
as $t>t_{0}$ big enough. Hence%
\[
\left\Vert P_{t_{0}}a_{i}P_{t_{0}}-P_{t_{0}}H_{i}\left(  P_{t}x_{1}%
P_{t},\cdots,P_{t}x_{n}P_{t}\right)  P_{t_{0}}\right\Vert
\]%
\[
=\left\Vert P_{t_{0}}P_{t}a_{i}P_{t}P_{t_{0}}-P_{t_{0}}H_{i}\left(  P_{t}%
x_{1}P_{t},\cdots,P_{t}x_{n}P_{t}\right)  P_{t_{0}}\right\Vert <L\cdot
\varepsilon
\]
Note that $\varepsilon$ is arbitrary, $\left\{  P_{t_{0}}a_{1}P_{t_{0}}%
,\cdots,P_{t_{0}}a_{M_{t_{0}}}P_{t_{0}}\right\}  $ is a linear basis of
$P_{t_{0}}\mathcal{A}P_{t_{0}}$ and
\[
P_{t_{0}}C^{\ast}(P_{t}x_{1}P_{t},\cdots,P_{t}x_{n}P_{t})P_{t_{0}}\subseteq
P_{t_{0}}\mathcal{A}P_{t_{0}},
\]
then
\[
\left\{  P_{t_{0}}H_{1}\left(  P_{t}x_{1}P_{t},\cdots,P_{t}x_{n}P_{t}\right)
P_{t_{0}},\cdots,P_{t_{0}}H_{M_{t_{o}}}\left(  P_{t}x_{1}P_{t},\cdots
,P_{t}x_{n}P_{t}\right)  P_{t_{0}}\right\}
\]
is a linearly independent family in $P_{t_{0}}C^{\ast}(P_{t}x_{1}P_{t}%
,\cdots,P_{t}x_{n}P_{t})P_{t_{0}}$ as $t$ big enough. Therefore
\[
\dim P_{t_{0}}C^{\ast}(P_{t}x_{1}P_{t},\cdots,P_{t}x_{n}P_{t})P_{t_{0}%
}=M_{t_{0}}.
\]
It follows that, for such $t,$
\begin{align*}
\dim C^{\ast}(P_{t}x_{1}P_{t},\cdots,P_{t}x_{n}P_{t})  &  \geq\dim P_{t_{0}%
}C^{\ast}(P_{t}x_{1}P_{t},\cdots,P_{t}x_{n}P_{t})P_{t_{0}}\\
&  =M_{t_{0}}=\dim\left(  P_{t_{0}}\mathcal{A}P_{t_{0}}\right)
\end{align*}
Let $N_{t}=\dim C^{\ast}(P_{t}x_{1}P_{t},\cdots,P_{t}x_{n}P_{t}).$ Then
\[
N_{t}=\dim C^{\ast}(P_{t}x_{1}P_{t},\cdots,P_{t}x_{n}P_{t})\rightarrow
\mathcal{1}%
\]
as $t\rightarrow\mathcal{1}$ since $M_{t}=\dim\left(  P_{t}\mathcal{A}%
P_{t}\right)  \rightarrow\mathcal{1}$ as $t\rightarrow\mathcal{1}$. Note that
\[
\delta_{top}\left(  P_{t}x_{1}P_{t},\cdots,P_{t}x_{n}P_{t}\right)  =1-\frac
{1}{N_{t}}.
\]
Then, for every $t$, there exists $\omega_{t}$ such that
\begin{align*}
&  1-\frac{1}{N_{t}-1}\text{ }\\
&  \leq\inf_{\varepsilon,r}\underset{k\rightarrow\mathcal{1}}{\lim\sup}%
\frac{v_{\left\Vert \cdot\right\Vert }\left(  \Gamma_{1/2}^{\text{top}}\left(
P_{t}x_{1}P_{t},\cdots,P_{t}x_{n}P_{t};k,\varepsilon,p_{1},\cdots
,p_{r}\right)  ,\omega_{t}\right)  }{-k^{2}\log\omega_{t}}.
\end{align*}
Then, for every $t,$ we have
\begin{align*}
&  1-\frac{1}{N_{t}-1}\\
&  \leq\inf_{\varepsilon_{0},r_{0}}\underset{k\rightarrow\mathcal{1}}{\lim
\sup}\frac{v_{\left\Vert \cdot\right\Vert }\left(  \Gamma_{1/2}^{\text{top}%
}(x_{1},\cdots,x_{n};k,\varepsilon_{0},p_{1},\cdots,p_{r_{0}}),\omega
_{t}\right)  }{-k^{2}\log\omega_{t}}%
\end{align*}
for $\varepsilon_{0},r_{0}$ by (4.1). Therefore%
\begin{align*}
&  \lim_{t\rightarrow\mathcal{1}}\left(  1-\frac{1}{N_{t}-1}\right) \\
&  \leq\underset{\omega_{t}}{\lim\sup}\inf_{\varepsilon_{0},r_{0}}%
\underset{k\rightarrow\mathcal{1}}{\lim\sup}\frac{v_{\left\Vert \cdot
\right\Vert }\left(  \Gamma_{1/2}^{\text{top}}(x_{1},\cdots,x_{n}%
;k,\varepsilon_{0},p_{1},\cdots,p_{r_{0}}),\omega_{t}\right)  }{-k^{2}%
\log\omega}\\
&  \leq\underset{\omega\rightarrow0}{\lim\sup}\inf_{\varepsilon_{0},r_{0}%
}\underset{k\rightarrow\mathcal{1}}{\lim\sup}\frac{v_{\left\Vert
\cdot\right\Vert }\left(  \Gamma_{1/2}^{\text{top}}(x_{1},\cdots
,x_{n};k,\varepsilon_{0},p_{1},\cdots,p_{r_{0}}),\omega\right)  }{-k^{2}%
\log\omega}\\
&  =\delta_{top}\left(  x_{1},\cdots,x_{n}\right)
\end{align*}

Hence $\delta_{top}\left(  x_{1},\cdots,x_{n}\right)  =1=1-\frac{1}%
{\dim\mathcal{A}}$ as $\dim\mathcal{A=1}$. This completes the proof.
\end{proof}

\begin{remark}
From \cite{BK2}, we know that a unital C*-algebra $\mathcal{A}=C^{\ast}\left(
x_{1},\cdots,x_{n}\right)  $ is strong NF if and only if it is nuclear and
inner QD. Therefore we have the following corollary:
\end{remark}

\begin{corollary}
Let $\mathcal{A}=C^{\ast}\left(  x_{1},\cdots,x_{n}\right)  $ be a strong NF
C*-algebra. Then $\delta_{top}\left(  x_{1},\cdots,x_{n}\right)  =1-\frac
{1}{\dim\mathcal{A}}$.
\end{corollary}

\section{\bigskip Definition and properties of $\mathfrak{K}_{top}^{(3)}$}

Suppose $\mathcal{A}$ is a unital C*-algebra.

Let $\mathcal{1\cdot}0=0.$ For any subset $\mathcal{G\subseteq A}$, define
\begin{align*}
\mathfrak{K}_{top}^{(3)}(x_{1},\cdots,x_{n};\mathcal{G)}  &  \mathcal{=}%
\inf\left\{  \mathcal{1\cdot}\mathfrak{K}_{top}^{(2)}(x_{1},\cdots,x_{n}%
:y_{1},\cdots,y_{t}):\left\{  y_{1},\cdots,y\right\}  \text{ is a finite
subset of }\mathcal{G}\right\}  ,\\
\mathfrak{K}_{top}^{(3)}(\mathcal{G)}  &  \mathcal{=}\sup
_{\substack{E\subseteq\mathcal{G}\\E\text{ is finite}}}\inf
_{\substack{F\subseteq\mathcal{G}\\F\text{ is finite}}}\mathfrak{K}%
_{top}^{(3)}(E:F)
\end{align*}

\begin{remark}
When $\mathcal{G}$ is finite, it is not difficult to see that%
\[
\mathfrak{K}_{top}^{(3)}(x_{1},\cdots,x_{n};\mathcal{G)=1\cdot}\mathfrak{K}%
_{top}^{(2)}(x_{1},\cdots,x_{n};\mathcal{G)}%
\]
and
\[
\mathfrak{K}_{top}^{(3)}(\mathcal{G)=1\cdot}\mathfrak{K}_{top}^{(2)}%
(\mathcal{G)}%
\]

\end{remark}

The proof of the following theorem is similar to the Theorem 3.3 in \cite{HW},
so we omit it.

\begin{theorem}
\label{hi}If $\mathcal{A}$ is an MF-algebra, then the following are equivalent:

\begin{enumerate}
\item $\mathfrak{K}_{top}^{(3)}(\mathcal{A)}=0;$

\item if $x_{1},\cdots,x_{n}\in\mathcal{A}$, then there exist $y_{1}%
,\cdots,y_{t}\in\mathcal{A}$ such that $\mathfrak{K}_{top}^{(2)}(x_{1},\cdots
x_{n}:y_{1},\cdots,y_{t})=0;$

\item for any generating set $\mathcal{G}$ of $\mathcal{A}$, $\mathfrak{K}%
_{top}^{(3)}(\mathcal{G)}=0;$

\item there exists a generating set $\mathcal{G}$ of $\mathcal{A}$ such that
$\mathfrak{K}_{top}^{(3)}(\mathcal{G)}=0;$

\item if $\mathcal{G}$ is a generating set of $\mathcal{A}$, and $A_{0}$ is a
finite subset of $\mathcal{G}$, then, for any finite subset $A$ with
$A_{0}\subseteq A\subseteq\mathcal{G}$, there exists a finite subset $B$ of
$\mathcal{G}$ so that $\mathfrak{K}_{top}^{(3)}(A:B)=0.$

\item there is an increasing directed family $\left\{  \mathcal{A}_{i}%
:i\in\Lambda\right\}  $ of C*-subalgebras of $\mathcal{A}$ such that

\begin{enumerate}
\item each $\mathcal{A}_{i}$ is countably generated;

\item $\mathfrak{K}_{top}^{(3)}(\mathcal{A}_{i}\mathcal{)}=0;$

\item $\mathcal{A=\cup}_{i\in\Lambda}\mathcal{A}_{i}$
\end{enumerate}

\item If $A$ is a countable subset of $\mathcal{M}$, then there exists a
countably generated subalgebra $\mathcal{B}$ of $\mathcal{M}$ such that
$A\subseteq\mathcal{B}$ and $\mathfrak{K}_{top}^{(3)}(\mathcal{B)=}0$
\end{enumerate}
\end{theorem}

\begin{remark}
\label{li3}If $\mathcal{A}$ is finite generated, then $\mathfrak{K}%
_{top}^{\left(  3\right)  }\left(  \mathcal{A}\right)  =0$ if
\[
\mathfrak{K}_{top}^{\left(  2\right)  }\left(  \mathcal{A}\right)  =0.
\]

\end{remark}

\begin{corollary}
Suppose $\mathcal{A}$ is a C*-algebra, $\mathcal{G}$ is a generating set of
$\mathcal{A}$. Then $\mathfrak{K}_{top}^{\left(  3\right)  }\left(
\mathcal{A}\right)  =\mathfrak{K}_{top}^{\left(  3\right)  }\left(
\mathcal{G}\right)  .$
\end{corollary}

\begin{corollary}
\label{7}Suppose $\left\{  \mathcal{A}_{l}\right\}  _{l\in\Lambda}$ is an
increasingly directed family of C*-algebras. Then
\[
\mathfrak{K}_{top}^{\left(  3\right)  }\left(  \cup\mathcal{A}_{l}\right)
\leq\lim\inf_{l}\mathfrak{K}_{top}^{\left(  3\right)  }\left(  \mathcal{A}%
_{l}\right)  .
\]

\end{corollary}

\begin{proof}
If $\lim\inf_{l}\mathfrak{K}_{top}^{\left(  3\right)  }\left(  \mathcal{A}%
_{l}\right)  =\mathcal{1}$, the inequality holds clearly. Suppose that
$\lim\inf_{l}\mathfrak{K}_{top}^{\left(  3\right)  }\left(  \mathcal{A}%
_{l}\right)  =0.$ Let $x_{1},\cdots,x_{n}\in\cup\mathcal{A}_{l}.$ Then we can
find a $\lambda\in\Lambda$ such that $x_{1},\cdots,x_{n}\in\mathcal{A}%
_{\lambda}$ and $\mathfrak{K}_{top}^{\left(  3\right)  }\left(  \mathcal{A}%
_{l}\right)  =0.$ Therefore, we can find $y_{1},\cdots,y_{p}\in\mathcal{A}%
_{\lambda}$ with
\[
\mathfrak{K}_{top}^{\left(  3\right)  }\left(  x_{1},\cdots,x_{n}:y_{1}%
,\cdots,y_{p}\right)  =0.
\]
It follows that $\mathfrak{K}_{top}^{\left(  3\right)  }\left(  \cup
\mathcal{A}_{l}\right)  =0$ by Theorem \ref{hi} (6).
\end{proof}

We modify the $\mathfrak{K}_{3}$ in \cite{HW} by using the modified free orbit
dimension $\mathfrak{K}_{2}^{(2)}.$

\begin{definition}
Let $\mathcal{M}$ be a von Neumann algebra with a tracial state $\tau,$ and
let $x_{1},\cdots,x_{n}$ be self-adjoint elements in $\mathcal{M}$. Define
\begin{align*}
&  \mathfrak{K}_{3}^{(3)}(x_{1},\cdots,x_{n};\mathcal{G)}\\
&  \mathcal{=}\inf\left\{  \mathcal{1\cdot}\mathfrak{K}_{2}^{(2)}(x_{1}%
,\cdots,x_{n}:y_{1},\cdots,y_{t}):\left\{  y_{1},\cdots,y\right\}  \text{ is a
finite subset of }\mathcal{G}\right\}
\end{align*}
and%
\[
\mathfrak{K}_{3}^{(3)}(\mathcal{G)}\mathcal{=}\sup_{\substack{E\subseteq
\mathcal{G}\\E\text{ is finite}}}\inf_{\substack{F\subseteq\mathcal{G}%
\\F\text{ is finite}}}\mathfrak{K}_{3}^{(3)}(E:F)
\]

\end{definition}

\begin{remark}
\label{88}Note that $\mathfrak{K}_{3}\left(  \mathcal{G}\right)  =0,$ then
$\mathfrak{K}_{3}^{\left(  3\right)  }\left(  \mathcal{G}\right)  =0$ by
Remark \ref{inddd}.
\end{remark}

\begin{theorem}
\label{k3top}Let $\mathcal{A}$ be an MF-algebra and $\mathcal{G}$ be a
generating set of $\mathcal{A}$. Then
\[
\mathfrak{K}_{top}^{\left(  3\right)  }\left(  \mathcal{G}\right)  \leq
\sup_{\tau\in\mathcal{TS}\left(  \mathcal{A}\right)  }\mathfrak{K}_{3}%
^{(3)}(\mathcal{G}:\tau)
\]

\end{theorem}

\begin{proof}
If $\mathfrak{K}_{top}^{\left(  3\right)  }\left(  \mathcal{G}\right)  =0,$
the inequality holds clearly. Now suppose $\mathfrak{K}_{top}^{\left(
3\right)  }\left(  \mathcal{G}\right)  =\mathcal{1},$then there is a subset
$E\subseteq\mathcal{G}$ such that for every finite subset $F\subseteq
\mathcal{G}$, $\mathfrak{K}_{top}^{\left(  3\right)  }(E:F)=\mathcal{1}$.
Therefore $0<\mathfrak{K}_{top}^{\left(  2\right)  }(E:F)\leq\mathcal{1}$ for
every finite subset $F\subseteq\mathcal{G}$. Let $F$ be an arbitrary finite
subset of $\mathcal{G}$. Then we can find a sequence $\left\{  F_{i}\right\}
_{i=1}^{\mathcal{1}}$ of finite subset of $\mathcal{G}$ with
\[
F=F_{0}\subseteq F_{1}\subseteq\cdots
\]
and $\cup_{i}F_{i}=\mathcal{G}$. Hence
\[
C^{\ast}\left(  E\cup F_{0}\right)  \subseteq C^{\ast}\left(  E\cup
F_{1}\right)  \subseteq\cdots
\]
and%
\[
\overline{\cup_{i}C^{\ast}\left(  E\cup F_{i}\right)  }^{\left\Vert
\cdot\right\Vert }=\mathcal{A}.
\]

\ \ \ Now for each $i,$ $\mathfrak{K}_{top}^{\left(  2\right)  }(E:F_{i})>0$.
So by Theorem \ref{li}, we may find a tracial state $\tau_{i}$ on $C^{\ast
}\left(  E\cup F_{i}\right)  $ such that
\[
\mathfrak{K}_{2}^{\left(  2\right)  }(E:F_{i};\tau_{i})\geq\mathfrak{K}%
_{top}^{\left(  2\right)  }(E:F_{i})-\varepsilon_{i}>0
\]
for some small $\varepsilon_{i}>0$ for each $i\in\mathbb{N}.$ If we regard
$C^{\ast}\left(  E\cup F_{0}\right)  $ as a subalgebra of $C^{\ast}\left(
E\cup F_{i}\right)  $ for every $i\geq1$, then
\begin{equation}
\mathfrak{K}_{2}^{(2)}(E:F_{0}:\tau_{i}\mathcal{)}\geq\mathfrak{K}%
_{2}^{\left(  2\right)  }(E:F_{i};\tau_{i})>0\text{ for every }i\in
\mathbb{N}\text{.} \tag{5.1}\label{(1)}%
\end{equation}
Let
\[
\pi:\overline{\cup_{i}C^{\ast}\left(  E\cup F_{i}\right)  }^{\left\Vert
\cdot\right\Vert }\longrightarrow\Pi_{i\in\mathbb{N}}C^{\ast}\left(  E\cup
F_{i}\right)  /\Sigma_{i\in\mathbb{N}}C^{\ast}\left(  E\cup F_{i}\right)
\]
be the embedding defined by
\[
\pi\left(  A\right)  =\underset{\text{the first }i\text{ positions}%
}{\underbrace{(0,\cdots,0,}}A,A,\cdots)\text{ for every }A\in C^{\ast}\left(
E\cup F_{i}\right)  .
\]
and $\widetilde{\tau}$ be the tracial state define by $\widetilde{\tau}\left(
\left[  A_{i}\right]  _{i}\right)  =\lim_{i\longrightarrow\alpha}\tau
_{i}(A_{i})$ on
\[
\Pi_{i\in\mathbb{N}}C^{\ast}\left(  E\cup F_{i}\right)  /\Sigma_{i\in
\mathbb{N}}C^{\ast}\left(  E\cup F_{i}\right)
\]
where $\alpha$ is a free ultrafilter. Define $\tau=\widetilde{\tau}\circ\pi$,
then $\tau$ is a tracial state on $\mathcal{A}.$ Note for any finite subset
$G$ of $\mathcal{G}$, we can always find a suitable index $i$ such that
\[
G\subseteq F_{i}\subseteq F_{i+1}\subseteq\cdots.
\]
Therefore $\widetilde{\tau}$ is irrelevant to the selection of finite subset
$F=F_{0},$ so is $\tau.$ Let $\left\{  \varepsilon_{t}\right\}  $ be a
decreasing sequence of positive number with $\lim_{t\rightarrow\mathcal{1}%
}\varepsilon_{t}=0$ and $\left\{  m_{t}\right\}  _{t=1}^{\mathcal{1}}$ be an
increasing sequence of integers with $\lim_{t\rightarrow\mathcal{1}}%
m_{t}=\mathcal{1}$. Then, for every $R>0,$ we can find a subsequence $\left\{
i_{t}\right\}  _{t=1}^{\mathcal{1}}$ of integers such that when $\varepsilon$
is small enough and $m$ is big enough, we always have
\[
\Gamma_{R}\left(  E:F_{0};k,\varepsilon,m;\tau_{i_{t}}\right)  \subseteq
\Gamma_{R}\left(  E:F_{0};k,\varepsilon_{t},m_{t};\tau\right)  \text{ .}%
\]
It implies that, for any $\omega>0,$%
\begin{align*}
&  \inf_{\varepsilon,m}\underset{k\rightarrow\mathcal{1}}{\lim\sup}\frac
{\log\left(  o_{2}\left(  \Gamma_{R}\left(  E:F_{0};k,\varepsilon
,m;\tau_{i_{t}}\right)  ,\omega\right)  \right)  }{k^{2}}\\
&  \leq\underset{k\rightarrow\mathcal{1}}{\lim\sup}\frac{\log\left(
o_{2}\left(  \Gamma_{R}\left(  E:F_{0};k,\varepsilon_{t},m_{t};\tau\right)
,\omega\right)  \right)  }{k^{2}}\text{ ,}%
\end{align*}
it follows that%
\begin{align*}
&  \underset{t\rightarrow\mathcal{1}}{\lim\sup}\inf_{\varepsilon,m}%
\underset{k\rightarrow\mathcal{1}}{\lim\sup}\frac{\log\left(  o_{2}\left(
\Gamma_{R}\left(  E:F_{0};k,\varepsilon,m;\tau_{i_{t}}\right)  ,\omega\right)
\right)  }{k^{2}}\\
&  \leq\underset{t\rightarrow\mathcal{1}}{\lim\sup}\underset{k\rightarrow
\mathcal{1}}{\lim\sup}\frac{\log\left(  o_{2}\left(  \Gamma_{R}\left(
E:F_{0};k,\varepsilon_{t},m_{t};\tau\right)  ,\omega\right)  \right)  }{k^{2}%
}\text{ ,}%
\end{align*}
Therefore we can find an index $t_{0}$ such that
\begin{align*}
0  &  <\mathfrak{K}_{2}^{(2)}(E:F_{0}:\tau_{i_{t_{0}}}\mathcal{)=}%
\sup_{0<\omega<1}\sup_{R>0}\inf_{\varepsilon,m}\underset{k\rightarrow
\mathcal{1}}{\lim\sup}\frac{\log\left(  o_{2}\left(  \Gamma_{R}\left(
E:F_{0};k,\varepsilon,m;\tau_{i_{t_{0}}}\right)  ,\omega\right)  \right)
}{k^{2}}\\
&  \mathcal{\leq}\sup_{0<\omega<1}\sup_{R>0}\underset{t\rightarrow\mathcal{1}%
}{\lim\sup}\inf_{\varepsilon,m}\underset{k\rightarrow\mathcal{1}}{\lim\sup
}\frac{\log\left(  o_{2}\left(  \Gamma_{R}\left(  E:F_{0};k,\varepsilon
,m;\tau_{i_{t}}\right)  ,\omega\right)  \right)  }{k^{2}}\\
&  \leq\sup_{0<\omega<1}\sup_{R>0}\underset{t\rightarrow\mathcal{1}}{\lim\sup
}\underset{k\rightarrow\mathcal{1}}{\lim\sup}\frac{\log\left(  o_{2}\left(
\Gamma_{R}\left(  E:F_{0};k,\varepsilon_{t},m_{t};\tau\right)  ,\omega\right)
\right)  }{k^{2}}\text{ }\\
&  =\sup_{0<\omega<1}\sup_{R>0}\inf_{\varepsilon_{t},m_{t}}\underset
{k\rightarrow\mathcal{1}}{\lim\sup}\frac{\log\left(  o_{2}\left(  \Gamma
_{R}\left(  E:F_{0};k,\varepsilon_{t},m_{t};\tau\right)  ,\omega\right)
\right)  }{k^{2}}\\
&  =\mathfrak{K}_{2}^{(2)}(E:F_{0}:\tau\mathcal{)}%
\end{align*}
Note that $F_{0}=F$ is an arbitrary subset of $\mathcal{G}$. It implies that
\begin{align*}
\mathfrak{K}_{3}^{(3)}(\mathcal{G}  &  :\tau)=\sup_{\substack{E\subseteq
\mathcal{G}\\E\text{ is finite}}}\inf_{\substack{F\subseteq\mathcal{G}%
\\F\text{ is finite}}}\mathfrak{K}_{3}^{(3)}(E:F;\tau)\\
&  =\sup_{\substack{E\subseteq\mathcal{G}\\E\text{ is finite}}}\inf
_{\substack{F\subseteq\mathcal{G}\\F\text{ is finite}}}\left\{
\mathcal{1\cdot}\mathfrak{K}_{2}^{(2)}(E:F:\tau\mathcal{)}\right\}
\mathcal{=1}\text{.}%
\end{align*}
It follows that
\[
\mathfrak{K}_{top}^{\left(  3\right)  }\left(  \mathcal{G}\right)  \leq
\sup_{\tau\in\mathcal{TS}\left(  \mathcal{A}\right)  }\mathfrak{K}_{3}%
^{(3)}(\mathcal{G}:\tau).
\]

\end{proof}

\begin{lemma}
(\cite{HS2})\label{1} Suppose $x$ is a self-adjoint element in a unital
C*-algebra $\mathcal{A}$. Let $\sigma(x)$ be the spectrum of x in
$\mathcal{A}$ and $R>\left\Vert x\right\Vert .$ For any $\varepsilon>0,$ we
have the following results.

\begin{enumerate}
\item There are some integer $n\geq1$ and distinct real numbers $\lambda
_{1},\cdots,\lambda_{n}$ in $\sigma\left(  x\right)  $ satisfying

\begin{enumerate}
\item $\left\vert \lambda_{i}-\lambda_{j}\right\vert \geq\omega$ for all
$1\leq i\neq j\leq n;$ and

\item for any $\lambda\in\sigma\left(  x\right)  ,$ there is some $\lambda
_{j}$ with $1\leq j\leq n$ such that $\left\vert \lambda-\lambda
_{j}\right\vert \leq\omega.$
\end{enumerate}

\item There is some $r_{0}\in\mathbb{N}$ such that the following holds: when
$r\geq r_{0}$ for any $k\in\mathbb{N}$ and any $A\in\Gamma_{R}^{\left(
top\right)  }\left(  x;k,r,\frac{1}{r}\right)  ,$ there are positive integers
$1\leq k_{1},\cdots,k_{n}\leq k$ with $k_{1}+\cdots+k_{n}=k$ and a unitary
matrix $U\in\mathcal{U}\left(  k\right)  $ satisfying
\[
\left\Vert U^{\ast}AU-\left(
\begin{array}
[c]{cccc}%
\lambda_{1}I_{k_{1}} & 0 & 0 & 0\\
0 & \lambda_{2}I_{k_{2}} & 0 & 0\\
0 & 0 & \ddots & 0\\
0 & 0 & 0 & \lambda_{n}I_{k_{n}}%
\end{array}
\right)  \right\Vert \leq2\omega
\]
where $I_{k_{j}}$ is the $k_{j}\times k_{j}$ identity matrix in $\mathcal{M}%
_{k_{j}}\left(  \mathbb{C}\right)  $ for $1\leq j\leq n.$
\end{enumerate}
\end{lemma}

\begin{lemma}
(\cite{DH})\label{2} If $\mathbb{B}$ is the unit ball in $\mathbb{R}^{m},$
then, with respect to the Euclidean metric
\[
\left(  \frac{1}{\omega}\right)  ^{m}\leq v\left(  \mathbb{B},\omega\right)
\leq\left(  \frac{3}{\omega}\right)  ^{m}%
\]

\end{lemma}

\begin{lemma}
\label{new}Let $\left\{  \lambda_{1},\cdots,\lambda_{k}\right\}
\subseteq\lbrack0,1]$ with $\lambda_{1}=\frac{1}{2k}$ and $\left\vert
\lambda_{i}-\lambda_{i+1}\right\vert =\frac{1}{k}$for $i=1,\cdots,k-1$. Assume
$D_{1}$ and $D_{2}$ are diagonal matrices in $\mathcal{M}_{k}\left(
\mathbb{C}\right)  $ with diagonal entries are all from $\left\{  \lambda
_{1},\cdots,\lambda_{k}\right\}  $ without repetition. For every $\delta>0,$
let%
\[
\Omega(D_{1},D_{2};\delta)=\left\{  U\in\mathcal{U}\left(  k\right)
|\left\Vert UD_{1}-D_{2}U\right\Vert _{2}\leq\delta\right\}  .
\]
Then, for every $0<\delta<r,$ there exists a set $\left\{  Ball\left(
U_{\lambda};\frac{4\delta}{r}\right)  \right\}  _{\lambda\in\Lambda}$ of
$\frac{4\delta}{r}$-balls in $\mathcal{U}\left(  k\right)  $ that cover
$\Omega(D_{1},D_{2};\delta)$ with the cardinality of $\Lambda$ satisfying
$\left\vert \Lambda\right\vert \leq\left(  \frac{3r}{2\delta}\right)
^{4rk^{2}}$
\end{lemma}

\begin{proof}
Let $D=diag\left(  \lambda_{1},\cdots,\lambda_{k}\right)  .$ Then there exist
$W_{1},W_{2}\in\mathcal{U}\left(  k\right)  $ such that $D_{1}=W_{1}%
DW_{1}^{\ast}$ and $D_{2}=W_{2}DW_{2}^{\ast}.$ Let
\[
\widetilde{\Omega}(\delta)=\left\{  U\in\mathcal{U}\left(  k\right)
|\left\Vert UD-DU\right\Vert _{2}\leq\delta\right\}  .
\]
Clearly
\[
\Omega(D_{1},D_{2};\delta)=\left\{  W_{2}^{\ast}UW_{1}|U\in\widetilde{\Omega
}(\delta)\right\}  ,
\]
hence $\widetilde{\Omega}(\delta)$ and $\Omega(D_{1},D_{2};\delta)$ have the
same covering numbers.

Let $\left\{  e_{st}\right\}  _{s,t=1}^{k}$ be the canonical system of matrix
units of $\mathcal{M}_{k}\left(  \mathbb{C}\right)  .$ Let
\[
\mathcal{S}_{1}=span\left\{  e_{st}:\left\vert \lambda_{s}-\lambda
_{t}\right\vert <r\right\}  ,\mathcal{S}_{2}=\mathcal{M}_{k}\left(
\mathbb{C}\right)  \ominus\mathcal{S}_{1}.
\]
For every $U=\sum_{s,t=1}^{k}x_{st}e_{st}$ in $\widetilde{\Omega}(\delta),$
with $x_{st}\in\mathbb{C}$, let $T_{1}=\sum_{e_{st}\in\mathcal{S}_{1}}%
x_{st}e_{st}$ and $T_{2}=\sum_{e_{st}\in\mathcal{S}_{2}}x_{st}e_{st}%
\in\mathcal{S}_{2}.$ But
\begin{align*}
\delta^{2}  &  \geq\left\Vert UD-DU\right\Vert _{2}^{2}=\sum_{s,t=1}%
^{k}\left\vert \left(  \lambda_{s}-\lambda_{t}\right)  x_{st}\right\vert
^{2}\geq\sum_{e_{st}\in\mathcal{S}_{2}}\left\vert \left(  \lambda_{s}%
-\lambda_{t}\right)  x_{st}\right\vert ^{2}\\
&  \geq r^{2}\sum_{e_{st}\in\mathcal{S}_{2}}\left\vert x_{st}\right\vert
^{2}=r^{2}\left\Vert T_{2}\right\Vert _{2}^{2}.
\end{align*}
Hence $\left\Vert T_{2}\right\Vert _{2}\leq\frac{\delta}{r}.$

Suppose $\lambda\in\lbrack0,1].$ The number of the points in $\left\{
\lambda_{1},\cdots,\lambda_{k}\right\}  $ which lie inside the interval
$\left(  \lambda-r,\lambda+r\right)  $ is at most the cardinality $m\neq0$ of
$\left(  \lambda-r,\lambda+r\right)  \cap\frac{1}{k}\mathbb{Z}$. Then
$\frac{1}{k}\leq2r,$ this interval contains at least one point in $\frac{1}%
{k}\mathbb{Z}$, and an interval of length $4r$ should contain at least $m+1$
such consecutive points. Hence $4r$ should be at least the length of the
interval defined by $m+1$( i.e., $m\frac{1}{k}$) consecutive points in
$\frac{1}{k}\mathbb{Z}$, i.e., $m\frac{1}{k}\leq4r.$ It follows that
$\dim_{\mathbb{R}}\mathcal{S}_{1}\leq4rk^{2}.$ Note that $\left\Vert
T_{1}\right\Vert _{2}\leq\left\Vert U\right\Vert _{2}=1.$ Then $\widetilde
{\Omega}(\delta)$ can be covered by a set $\left\{  Ball\left(  A^{\lambda
};\frac{2\delta}{r}\right)  \right\}  _{\lambda\in\Lambda}$ of $\frac{2\delta
}{r}$-balls in $\mathcal{M}_{k}\left(  \mathbb{C}\right)  $ with $\left\vert
\Lambda\right\vert \leq\left(  \frac{3r}{2\delta}\right)  ^{4rk^{2}}$ by Lemma
\ref{2}. Because $\widetilde{\Omega}(\delta)\subseteq\mathcal{U}\left(
k\right)  ,$ after replacing $A^{\lambda}$ by a unitary $U^{\lambda}$ in
$Ball\left(  A^{\lambda},\frac{2\delta}{r}\right)  ,$ we obtain that the set
$\left\{  Ball\left(  U_{\lambda};\frac{4\delta}{r}\right)  \right\}
_{\lambda\in\Lambda}$ of $\frac{4\delta}{r}$-balls in $\mathcal{U}\left(
k\right)  $ that cover $\widetilde{\Omega}(\delta)$ with the cardinality of
$\Lambda$ satisfying $\left\vert \Lambda\right\vert \leq\left(  \frac
{3r}{2\delta}\right)  ^{4rk^{2}}.$ The same result holds for $\Omega
(D_{1},D_{2};\delta).$
\end{proof}

The following lemma can be found in \cite{DH}:

\begin{lemma}
\label{3}Suppose $A$ is a normal element in a von Neumann algebra
$\mathcal{M}$ with tracial state $\tau$ such that $A$ has no eigenvalues. Then
there is a positive element $Y$ with the uniform distribution on $[0,1]$ such
that $W^{\ast}\left(  A\right)  =W^{\ast}\left(  Y\right)  .$
\end{lemma}

\begin{remark}
\label{96}It is well-known that every selfadjoint element of a finite von
Neumann algebra $\mathcal{M}$ has an eigenvalue if and only if $\mathcal{M}$
has a finite-dimensional invariant subspace.
\end{remark}

\begin{lemma}
\label{4}Let $x_{1},\cdots,x_{n},y_{1},\cdots,y_{p},v_{1},\cdots,v_{s}%
,w_{1},\cdots,w_{t}$ be elements in an MF C*-algebra $\mathcal{A}$. If
C*$\left(  x_{1},\cdots,x_{n}\right)  \cap C^{\ast}\left(  y_{1},\cdots
,y_{p}\right)  $ has no finite-dimensional representations, then
\begin{align*}
&  \mathfrak{K}_{top}^{\left(  3\right)  }\left(  x_{1},\cdots,x_{n}%
,y_{1},\cdots,y_{p}:v_{1},\cdots,v_{s},w_{1},\cdots,w_{t}\right) \\
&  \leq\mathfrak{K}_{top}^{\left(  3\right)  }\left(  x_{1},\cdots,x_{n}%
:v_{1},\cdots,v_{s}\right)  +\mathfrak{K}_{top}^{\left(  3\right)  }\left(
y_{1},\cdots,y_{p}:w_{1},\cdots,w_{t}\right)
\end{align*}

\end{lemma}

\begin{proof}
Without loss of generality, we may assume that $\left\Vert x_{i}\right\Vert
\leq1$ and $\left\Vert y_{j}\right\Vert \leq1$ for each $1\leq i\leq n$ and
$1\leq j\leq p.$ If one of $\mathfrak{K}_{top}^{\left(  3\right)  }\left(
x_{1},\cdots,x_{n}:v_{1},\cdots,v_{s}\right)  $ and $\mathfrak{K}%
_{top}^{\left(  3\right)  }\left(  y_{1},\cdots,y_{p}:w_{1},\cdots
,w_{t}\right)  $ is infinity, then we are done. Now suppose
\[
\mathfrak{K}_{top}^{\left(  3\right)  }\left(  x_{1},\cdots,x_{n}:v_{1}%
,\cdots,v_{s}\right)  =\mathfrak{K}_{top}^{\left(  3\right)  }\left(
y_{1},\cdots,y_{p}:w_{1},\cdots,w_{t}\right)  =0.
\]
Let $\tau$ be a trace on $C^{\ast}\left(  x_{1},\cdots,x_{n},y_{1}%
,\cdots,y_{p},v_{1},\cdots,v_{s},w_{1},\cdots,w_{t}\right)  .$ Let $(\pi,M,e)$
denote the GNS construction for $\tau,$ i.e., $\pi:\mathcal{A\longrightarrow
B}\left(  M\right)  $ is a unital *-homomorphism with unit cyclic vector $e$,
such that, for every $a\in\mathcal{A}$, we have $\tau(a)=(\pi(a)e,e).$ Note
the state $\rho$ defined by $\rho\left(  a\right)  =(\pi(a)e,e)$ is faithful
on
\[
\pi\left(  C^{\ast}\left(  x_{1},\cdots,x_{n}\right)  \cap C^{\ast}\left(
y_{1},\cdots,y_{p}\right)  \right)  ^{\prime\prime},
\]
so $\pi\left(  C^{\ast}\left(  x_{1},\cdots,x_{n}\right)  \cap C^{\ast}\left(
y_{1},\cdots,y_{p}\right)  \right)  ^{\prime\prime}$ is finite. Since
C*$\left(  x_{1},\cdots,x_{n}\right)  \cap C^{\ast}\left(  y_{1},\cdots
,y_{p}\right)  $ has no finite-dimensional representation,
\[
\pi\left(  C^{\ast}\left(  x_{1},\cdots,x_{n}\right)  \cap C^{\ast}\left(
y_{1},\cdots,y_{p}\right)  \right)  ^{\prime\prime}%
\]
has no non-zero finite-dimensional invariant subspace. Hence there is an
\[
a=a^{\ast}\in\pi\left(  C^{\ast}\left(  x_{1},\cdots,x_{n}\right)  \cap
C^{\ast}\left(  y_{1},\cdots,y_{p}\right)  \right)  ^{\prime\prime}%
\]
such that $a$ has no eigenvalues by Remark \ref{96}. Let $\left\{
a_{k}\right\}  $ be a sequence of self-adjoint elements in $\pi\left(
C^{\ast}\left(  x_{1},\cdots,x_{n}\right)  \cap C^{\ast}\left(  y_{1}%
,\cdots,y_{p}\right)  \right)  $ with $a_{k}\rightarrow a$ in the $\ast
$-strong operator topology. Then Voiculescu (\cite{DV2}) proved that
\[
1=\delta_{0}(a,\rho)\leq\underset{k\rightarrow\mathcal{1}}{\lim\inf}\delta
_{0}(a_{k},\rho)=\underset{k\rightarrow\mathcal{1}}{\lim\inf}\left(
1-\sum_{t\text{ is an eigenvalue of a}}\rho\left(  P_{t}\right)  ^{2}\right)
\]
where $\delta_{0}$ is the free entropy dimension for von Neumann algebras. So
we can find a self-adjoint element $a_{k}\in\pi\left(  C^{\ast}\left(
x_{1},\cdots,x_{n}\right)  \cap C^{\ast}\left(  y_{1},\cdots,y_{p}\right)
\right)  $ satisfying that $a_{k}$ has no eigenvalues. Therefore there is an
interval $[\alpha,\beta]$ such that $[\alpha,\beta]\subseteq\sigma\left(
a_{k}\right)  .$ Let $b$ be a self-adjoint element in $C^{\ast}\left(
x_{1},\cdots,x_{n}\right)  \cap C^{\ast}\left(  y_{1},\cdots,y_{p}\right)  $
satisfying $\pi\left(  b\right)  =a_{k}$. Then $[\alpha,\beta]\subseteq
\sigma\left(  b\right)  $. Define a continuous function
\[
f(t)=\left\{
\begin{array}
[c]{cc}%
0 & ,\text{ -}\left\Vert b\right\Vert \leq x<\alpha\\
\frac{x-\alpha}{\beta-\alpha} & \alpha\leq x\leq\beta\\
1 & \beta<x\leq\left\Vert b\right\Vert
\end{array}
\right.
\]
over $[-\left\Vert b\right\Vert ,\left\Vert b\right\Vert ].$ Then
$\sigma\left(  f(b)\right)  =[0,1]$ and
\[
f\left(  b\right)  \in C^{\ast}\left(  x_{1},\cdots,x_{n}\right)  \cap
C^{\ast}\left(  y_{1},\cdots,y_{p}\right)  .
\]
Denote $d=f(b)$. If%
\begin{align*}
&  \left(  A_{1},\cdots,A_{n},B_{1},\cdots,B_{p},D\right) \\
&  \in\Gamma^{\left(  top\right)  }(x_{1},\cdots,x_{n},y_{1},\cdots
,y_{p},d:v_{1},\cdots,v_{s},w_{1},\cdots,w_{t};P_{1},\cdots,P_{m}%
,k,\varepsilon),
\end{align*}
then
\[
\left(  A_{1},\cdots,A_{n},D\right)  \in\Gamma^{\left(  top\right)  }%
(x_{1},\cdots,x_{n},d:v_{1},\cdots,v_{s},w_{1},\cdots,w_{t};P_{1}^{\prime
},\cdots,P_{m_{1}}^{\prime},k,\varepsilon)
\]%
\[
\left(  B_{1},\cdots,B_{p},D\right)  \in\Gamma^{\left(  top\right)  }\left(
y_{1},\cdots,y_{p},d:v_{1},\cdots,v_{s},w_{1},\cdots,w_{t};P_{1}^{\prime
\prime},\cdots,P_{m_{2}}^{\prime\prime},k,\varepsilon\right)
\]
where $P_{1}^{\prime},\cdots,P_{m_{1}}^{\prime}\in\mathbb{C\langle}%
X_{1},\cdots,X_{n+1},V_{1},\cdots,V_{s},W_{1},\cdots,W_{t}\mathbb{\rangle}$
and $P_{1}^{\prime\prime},\cdots,P_{m_{2}}^{\prime\prime}\in\mathbb{C\langle
}Y_{1},\cdots,Y_{p+1},V_{1},\cdots,V_{s},W_{1},\cdots,W_{t}\mathbb{\rangle}$.
Let $\left\{  \mathcal{U(}A_{1}^{\lambda},\cdots,A_{n}^{\lambda},D^{\lambda
});\frac{r\omega}{24}\right\}  _{\lambda\in\Lambda_{k}}$ be a set of
$\frac{r\omega}{24}$-orbit-balls that cover
\[
\Gamma^{\left(  top\right)  }(x_{1},\cdots,x_{n},d:v_{1},\cdots,v_{s}%
,w_{1},\cdots,w_{t};k,\varepsilon,P_{1}^{\prime},\cdots,P_{m_{1}}^{\prime})
\]
with the cardinality of $\Lambda_{k}$ satisfying
\[
\left\vert \Lambda_{k}\right\vert =o_{2}\left(  \Gamma^{\left(  top\right)
}(x_{1},\cdots,x_{n},d:v_{1},\cdots,v_{s},w_{1},\cdots,w_{t};k,\varepsilon
,P_{1}^{\prime},\cdots,P_{m_{1}}^{\prime});\frac{r\omega}{24}\right)  .
\]
Also let $\left\{  \mathcal{U(}B_{1}^{\sigma},\cdots,B_{n}^{\sigma},D^{\sigma
});\frac{r\omega}{24}\right\}  _{\sigma\in\Sigma_{k}}$ be a set of
$\frac{r\omega}{24}$-orbit-balls that cover
\[
\Gamma_{R}\left(  y_{1},\cdots,y_{p},d:v_{1},\cdots,v_{s},w_{1},\cdots
,w_{t};k,\varepsilon,P_{1}^{\prime\prime},\cdots,P_{m_{2}}^{\prime\prime
}\right)
\]
with the cardinality of $\Sigma_{k}$ satisfying
\[
\left\vert \Sigma_{k}\right\vert =o_{2}\left(  \Gamma_{R}\left(  y_{1}%
,\cdots,y_{p},d:v_{1},\cdots,v_{s},w_{1},\cdots,w_{t};k,\varepsilon
,P_{1}^{\prime\prime},\cdots,P_{m_{2}}^{\prime\prime}\right)  ;\frac{r\omega
}{24}\right)  .
\]
When $m$ is sufficiently large and $\varepsilon$ is sufficiently small, we may
assume that all $D^{\sigma},D^{\lambda}$ are diagonal matrices in Lemma
\ref{new} by Lemma \ref{1}. For any
\[
\left(  A_{1},\cdots,A_{n},B_{1},\cdots,B_{p},D\right)
\]%
\[
\in\Gamma_{R}\left(  x_{1},\cdots,x_{n},y_{1},\cdots,y_{p},d:v_{1}%
,\cdots,v_{s},w_{1},\cdots,w_{t};P_{1},\cdots,P_{m},k,\varepsilon\right)  ,
\]
there exist some $\lambda\in\Lambda_{k},\sigma\in\Sigma_{k}$ and $W_{1}%
,W_{2}\in\mathcal{U}_{k}$ such that
\[
\left\Vert \left(  A_{1},\cdots,A_{n},D\right)  -W_{1}\left(  A_{1}^{\lambda
},\cdots,A_{n}^{\lambda},D^{\lambda}\right)  W_{1}^{\ast}\right\Vert _{2}%
\leq\frac{r\omega}{24},
\]%
\[
\left\Vert \left(  B_{1},\cdots,B_{p},D\right)  -W_{2}\left(  B_{1}^{\sigma
},\cdots,B_{p}^{\sigma},D^{\sigma}\right)  W_{2}^{\ast}\right\Vert _{2}%
\leq\frac{r\omega}{24}.
\]
Therefore%
\[
\left\Vert W_{1}D^{\lambda}W_{1}^{\ast}-W_{2}D^{\sigma}W_{2}^{\ast}\right\Vert
_{2}=\left\Vert W_{2}^{\ast}W_{1}D^{\lambda}-D^{\sigma}W_{2}^{\ast}%
W_{1}\right\Vert _{2}\leq\frac{r\omega}{12}.
\]
From Lemma \ref{new}, there exists a set $\left\{  Ball\left(  U_{\lambda
,\sigma,\gamma},\frac{\omega}{3}\right)  \right\}  _{\gamma\in\Delta_{k}}$ in
$\mathcal{U}\left(  k\right)  $ which cover $\Omega\left(  D^{\lambda
},D^{\sigma};\frac{r\omega}{12}\right)  $ with cardinality $\left\vert
\Delta_{k}\right\vert \leq\left(  \frac{18}{\omega}\right)  ^{4rk^{2}}.$ This
implies that
\begin{align*}
&  ||\left(  A_{1},\cdots,A_{n},B_{1},\cdots,B_{p},D\right) \\
&  -\left(  W_{2}U_{\lambda,\sigma,\gamma}A_{1}^{\lambda}U_{\lambda
,\sigma,\gamma}^{\ast}W_{2}^{\ast},\cdots,W_{2}U_{\lambda,\sigma,\gamma}%
A_{n}^{\lambda}U_{\lambda,\sigma,\gamma}^{\ast}W_{2}^{\ast},W_{2}B_{1}%
^{\sigma}W_{2}^{\ast},\cdots,W_{2}B_{p}^{\sigma}W_{2}^{\ast}\right)
||_{2}<n\omega
\end{align*}
Then we get
\begin{align*}
&  \mathfrak{K}_{3}^{(top)}\left(  x_{1},\cdots,x_{n},y_{1},\cdots
,y_{p},u:v_{1},\cdots,v_{s},w_{1},\cdots,w_{t};2n\omega\right) \\
&  \leq\inf_{m\in\mathbb{N}\text{,}\varepsilon>0}\underset{k\longrightarrow
\mathcal{1}}{\lim\sup}\frac{\log\left(  \left\vert \Lambda_{k}\right\vert
\left\vert \Sigma_{k}\right\vert \left\vert \Delta_{k}\right\vert \right)
}{k^{2}}\leq4r\left(  \log\left(  18\right)  -\log\omega\right)
\end{align*}

Because $r$ is an arbitrarily small number, we have
\[
\mathfrak{K}_{3}^{(top)}\left(  x_{1},\cdots,x_{n},y_{1},\cdots,y_{p}%
,d:v_{1},\cdots,v_{s},w_{1},\cdots,w_{t}\right)  =0.
\]
Note that $C^{\ast}\left(  x_{1},\cdots,x_{n},y_{1},\cdots,y_{p},d\right)
=C^{\ast}\left(  x_{1},\cdots,x_{n},y_{1},\cdots,y_{p}\right)  ,$ so by
Theorem \ref{ind}, we have
\[
\mathfrak{K}_{3}^{(top)}\left(  x_{1},\cdots,x_{n},y_{1},\cdots,y_{p}%
:v_{1},\cdots,v_{s},w_{1},\cdots,w_{t}\right)  =0.
\]
This completes the proof.
\end{proof}

\begin{theorem}
\label{6}Suppose $\mathcal{A}$ is a C*-algebra, $\mathcal{N}_{1}$ and
$\mathcal{N}_{2}$ are C*-subalgebra of $\mathcal{A}$. If $\mathcal{N}_{1}%
\cap\mathcal{N}_{2}$ is has no finite-dimensional representation, then%
\[
\mathfrak{K}_{3}^{(top)}\left(  C^{\ast}\left(  \mathcal{N}_{1}\cup
\mathcal{N}_{2}\right)  \right)  \leq\mathfrak{K}_{3}^{(top)}\left(
\mathcal{N}_{1}\right)  +\mathfrak{K}_{3}^{(top)}\left(  \mathcal{N}%
_{2}\right)
\]

\end{theorem}

\begin{proof}
If one of $\mathfrak{K}_{3}^{(top)}\left(  \mathcal{N}_{1}\right)  $ and
$\mathfrak{K}_{3}^{(top)}\left(  \mathcal{N}_{2}\right)  $ is infinity, the
inequality holds automatically.

Now suppose that
\[
\mathfrak{K}_{3}^{(top)}\left(  \mathcal{N}_{1}\right)  =\mathfrak{K}%
_{3}^{(top)}\left(  \mathcal{N}_{2}\right)  =0.
\]
By the same argument in the first part proof of Lemma \ref{4}, we can find a
self-adjoint element $d$ in $\mathcal{N}_{1}\cap\mathcal{N}_{2}$ with
$\sigma\left(  d\right)  =[0,1].$ Let $\mathcal{G=N}_{1}\cup\mathcal{N}_{2}$
and $A_{0}=\left\{  d\right\}  .$ Then $\mathcal{G}$ is a generating set of
$C^{\ast}\left(  \mathcal{N}_{1}\cup\mathcal{N}_{2}\right)  $. Suppose
$A_{0}\subseteq A\subseteq\mathcal{G}$ where $A=\left\{  x_{1},\cdots
,x_{n},d,y_{1},\cdots,y_{p}\right\}  $ is a finite subset with $x_{1}%
,\cdots,x_{n}\in\mathcal{N}_{1}$ and $y_{1},\cdots,y_{p}\in\mathcal{N}_{2}.$
Since $\mathfrak{K}_{3}^{(top)}\left(  \mathcal{N}_{1}\right)  =\mathfrak{K}%
_{3}^{(top)}\left(  \mathcal{N}_{2}\right)  =0,$ there exists
\[
v_{1},\cdots,v_{s}\in\mathcal{N}_{1},w_{1},\cdots,w_{t}\in\mathcal{N}_{2}%
\]
such that
\[
\mathfrak{K}_{2}^{(top)}\left(  x_{1},\cdots,x_{n},d:v_{1},\cdots
,v_{s}\right)  =\mathfrak{K}_{2}^{(top)}\left(  y_{1},\cdots,y_{t}%
,d:w_{1},\cdots,w_{t}\right)  =0.
\]
Because $d\in C^{\ast}\left(  x_{1},\cdots,x_{n},d\right)  \cap C^{\ast
}\left(  y_{1},\cdots,y_{p},d\right)  ,$ then from Lemma \ref{4}, we know
that
\begin{align*}
&  \mathfrak{K}_{3}^{(top)}\left(  A:v_{1},\cdots,v_{s},w_{1},\cdots
,w_{t}\right) \\
&  =\mathfrak{K}_{3}^{(top)}\left(  x_{1},\cdots,x_{n},y_{1},\cdots
,y_{p},d:v_{1},\cdots,v_{s},w_{1},\cdots,w_{t}\right)  =0
\end{align*}
Therefore, by Theorem \ref{hi} (5), $\mathfrak{K}_{3}^{(top)}\left(  C^{\ast
}\left(  \mathcal{N}_{1}\cup\mathcal{N}_{2}\right)  \right)  =0.$ This
completes the proof.
\end{proof}

\begin{theorem}
\label{qq}Let $\mathcal{N}$ be an MF C*-algebra and $\mathcal{A\subseteq N}$
be a C*-subalgebra such that $\mathcal{A}$ has a self-adjoint element $a$ with
$\sigma\left(  a\right)  =[0,1].$ If there is an unitary $u\in\mathcal{N}$
such that $u^{\ast}au\subseteq\mathcal{A}$. Then
\[
\mathfrak{K}_{3}^{(top)}\left(  C^{\ast}\left(  \mathcal{A}\cup\left\{
u\right\}  \right)  \right)  \leq\mathfrak{K}_{3}^{(top)}\left(
\mathcal{A}\right)
\]

\end{theorem}

\begin{proof}
If $\mathfrak{K}_{3}^{(top)}\left(  \mathcal{A}\right)  =\mathcal{1}$, we are
done. Now suppose that $\mathfrak{K}_{3}^{(top)}\left(  \mathcal{A}\right)
=0.$Let $x_{1},\cdots,x_{n},a,u^{\ast}au$ be elements in $\mathcal{A}$. Then
there exist $y_{1},\cdots,y_{p}$ in $\mathcal{A}$ such that
\[
\mathfrak{K}_{2}^{(top)}\left(  x_{1},\cdots,x_{n},a,u^{\ast}au:y_{1}%
,\cdots,y_{p}\right)  =0.
\]
For any $0<\omega<1,0<r<1,m,k\in\mathbb{N}$ and $\varepsilon>0,$ there exists
a set
\[
\left\{  \mathcal{U}\left(  T_{1}^{\lambda},\cdots,T_{n}^{\lambda},A^{\lambda
},B^{\lambda};\frac{r\omega}{64}\right)  \right\}  _{\lambda\in\Lambda_{k}}%
\]
of $\frac{r\omega}{64}$-orbit-balls in $\mathcal{M}_{k}\left(  \mathbb{C}%
\right)  ^{n+2}$ that cover
\[
\Gamma^{(top)}\left(  x_{1},\cdots,x_{n},a,u^{\ast}au:y_{1},\cdots
,y_{p}:k,\varepsilon,P_{1},\cdots,P_{m}\right)
\]
where $P_{1},\cdots,P_{m}\in\mathbb{C\langle}X_{1},\cdots,X_{n+2},Y_{1}%
,\cdots,Y_{p}\mathbb{\rangle}$ with the cardinality of $\Lambda_{k}$
satisfying
\[
\left\vert \Lambda_{k}\right\vert =o_{2}\left(  \Gamma^{(top)}\left(
x_{1},\cdots,x_{n},a,u^{\ast}au:y_{1},\cdots,y_{p};k,\varepsilon,P_{1}%
,\cdots,P_{m}\right)  ,\frac{r\omega}{64}\right)  .
\]
When $m$ is sufficiently large and $\varepsilon$ is sufficiently small, we can
assume that $A^{\lambda}$ is in the form in Lemma \ref{new}, and $B^{\lambda
}=U^{\ast}A^{\lambda}U$ for some unitary matrix $U.$

For sufficiently large $m^{\prime}$ and sufficiently small $\varepsilon\left(
\leq\frac{r\omega}{64}\right)  ,$ when
\begin{align*}
&  \left(  T_{1},\cdots,T_{n},A,B,C,D\right) \\
&  \in\Gamma^{(top)}\left(  x_{1},\cdots,x_{n},a,u^{\ast}au,\frac{u+u^{\ast}%
}{2},\frac{u-u^{\ast}}{2i}:y_{1},\cdots,y_{p};P_{1}^{\prime},\cdots
,P_{m^{\prime}}^{\prime},k,\varepsilon\right)  .
\end{align*}
where $P_{1}^{\prime},\cdots,P_{m^{\prime}}^{\prime}\in\mathbb{C\langle}%
X_{1},\cdots,X_{n+4},Y_{1},\cdots,Y_{p}\mathbb{\rangle}$. We may assume that
$\left\Vert C+iD\right\Vert \leq\frac{3}{2}$ as $m^{\prime}$ large enough and
$\varepsilon$ small enough. In addition, it is clear that
\[
\left\Vert A\left(  C+iD\right)  -\left(  C+iD\right)  B\right\Vert
\leq\varepsilon
\]
and%
\[
\left(  T_{1},\cdots,T_{n},A,B\right)  \in\Gamma^{(top)}\left(  x_{1}%
,\cdots,x_{n},a,u^{\ast}au:y_{1},\cdots,y_{p};k,\varepsilon,P_{1},\cdots
,P_{m}\right)  .
\]
for some $m$. So there exists some $\lambda\in\Lambda_{k}$ and $V\in
\mathcal{U}_{k}$ such that
\[
\left\Vert \left(  T_{1},\cdots,T_{n},A,B\right)  -\left(  VT_{1}^{\lambda
}V^{\ast},\cdots,VT_{n}^{\lambda}V^{\ast},VA^{\lambda}V^{\ast},VB^{\lambda
}V^{\ast}\right)  \right\Vert _{2}\leq\frac{r\omega}{64}.
\]
By Lemma \ref{new}, $Ball\left(  U^{\lambda};\frac{\omega}{2}\right)  $ of
$\frac{\omega}{2}$-balls in $\mathcal{U}\left(  k\right)  $ that cover
$\Omega\left(  A^{\lambda},B^{\lambda},\frac{r\omega}{8}\right)  $ with the
cardinality of $\Sigma_{k}$ satisfying $\left\vert \Sigma_{k}\right\vert
\leq\left(  \frac{12}{\omega}\right)  ^{4rk^{2}}.$ Since%
\begin{align*}
&  \left\Vert A\left(  C+iD\right)  -\left(  C+iD\right)  B\right\Vert _{2}\\
&  \leq\left\Vert A\left(  C+iD\right)  -\left(  C+iD\right)  B\right\Vert
_{2}\\
&  \leq\left\Vert A\left(  C+iD\right)  -\left(  C+iD\right)  B\right\Vert
\leq\varepsilon\leq\frac{r\omega}{64}\text{, }%
\end{align*}
Then%
\begin{align*}
&  \left\Vert V^{\ast}\left(  C+iD\right)  VA^{\lambda}-B^{\lambda}V^{\ast
}\left(  C+iD\right)  V\right\Vert _{2}\\
&  =\left\Vert \left(  C+iD\right)  VA^{\lambda}V^{\ast}-VB^{\lambda}V^{\ast
}\left(  C+iD\right)  \right\Vert _{2}\leq\frac{r\omega}{16}.
\end{align*}
Since $C+iD$ is very close to a unitary, then we may find an unitary
$U^{\lambda}\in\Omega\left(  A^{\lambda},B^{\lambda},\frac{r\omega}{8}\right)
$ such that
\[
\left\Vert V^{\ast}\left(  C+iD\right)  V-U^{\lambda}\right\Vert _{2}%
<\frac{r\omega}{4}.
\]
It implies that%
\begin{align*}
&  \left\Vert \left(  T_{1},\cdots,T_{n},C,D\right)  -\left(  VT_{1}^{\lambda
}V^{\ast},\cdots,VT_{n}^{\lambda}V^{\ast},V\frac{U^{\lambda}+U^{\lambda\ast}%
}{2}V^{\ast},V\frac{U^{\lambda}-U^{\lambda\ast}}{2i}V^{\ast}\right)
\right\Vert _{2}\\
&  \leq r\omega
\end{align*}

Therefore
\begin{align*}
&  o_{2}\left(  \Gamma^{(top)}\left(  x_{1},\cdots,x_{n},\frac{u+u^{\ast}}%
{2},\frac{u-u^{\ast}}{2i}:a,uau^{\ast},y_{1}\cdots,y_{p};k,\varepsilon
,P_{1},\cdots,P_{m}\right)  ,\omega\right) \\
&  \leq\left\vert \Lambda_{k}\right\vert \left\vert \Sigma_{k}\right\vert
\end{align*}
Hence, we get
\begin{align*}
0  &  \leq\mathfrak{K}_{2}^{(top)}\left(  x_{1},\cdots,x_{n},\frac{u+u^{\ast}%
}{2},\frac{u-u^{\ast}}{2i}:a,uau^{\ast},y_{1}\cdots,y_{p},\omega\right) \\
&  \leq\inf_{m\in\mathbb{N}\text{,}\varepsilon>0}\underset{k\longrightarrow
\mathcal{1}}{\lim\sup}\frac{\log\left(  \left\vert \Lambda_{k}\right\vert
\left\vert \Sigma_{k}\right\vert \right)  }{k^{2}}\\
&  \leq\inf_{m\in\mathbb{N}\text{,}\varepsilon>0}\underset{k\longrightarrow
\mathcal{1}}{\lim\sup}\left(  \frac{\log\left(  \left\vert \Lambda
_{k}\right\vert \right)  }{k^{2}}+4r(\log12-\log\omega)\right) \\
&  =4r(\log12-\log\omega)
\end{align*}
Since $r$ is an arbitrarily small positive number, we have
\[
\mathfrak{K}_{2}^{(top)}\left(  x_{1},\cdots,x_{n},\frac{u+u^{\ast}}{2}%
,\frac{u-u^{\ast}}{2i}:a,uau^{\ast},y_{1}\cdots,y_{p},\omega\right)  =0.
\]
Therefore
\[
\mathfrak{K}_{3}^{(top)}\left(  x_{1},\cdots,x_{n},\frac{u+u^{\ast}}{2}%
,\frac{u-u^{\ast}}{2i}:a,uau^{\ast},y_{1}\cdots,y_{p}\right)  =0.
\]
Hence
\[
\mathfrak{K}_{3}^{(top)}\left(  C^{\ast}\left(  \mathcal{A\cup}\left\{
u\right\}  \right)  \right)  =0.
\]

\end{proof}

\begin{theorem}
\label{5}Let $\mathcal{N}$ be an MF C*-algebra and $\mathcal{A\subseteq N}$ be
a C*-subalgebra where $\mathcal{A}$ has no finite-dimensional
representations$.$ If there is an unitary $u\in\mathcal{N}$ such that
$u\mathcal{A}u^{\ast}\subseteq\mathcal{A}$ for some . Then
\[
\mathfrak{K}_{3}^{(top)}\left(  C^{\ast}\left(  \mathcal{A}\cup\left\{
u\right\}  \right)  \right)  \leq\mathfrak{K}_{3}^{(top)}\left(
\mathcal{A}\right)
\]

\end{theorem}

\begin{proof}
If $\mathfrak{K}_{3}^{(top)}\left(  \mathcal{A}\right)  =\mathcal{1}$, the
inequality is clear. Now suppose that $\mathfrak{K}_{3}^{(top)}\left(
\mathcal{A}\right)  =0.$ Same as the first part proof of Lemma \ref{4}, we can
find an element $a\in\mathcal{A}$ satisfying $\sigma\left(  a\right)  =\left[
0,1\right]  $. Then the inequality holds by Theorem \ref{qq}.
\end{proof}

\begin{corollary}
Suppose $\mathcal{A}$ is a unital MF algebra with no finite-dimensional
representations and $\mathcal{B}$ is a unital MF algebra. Suppose $G$ is a
countable group of actions $\left\{  \alpha_{g}\right\}  _{g\in G}$ on
$\mathcal{A}$. Suppose $\mathcal{D=A\rtimes}_{\alpha}G$ is either a full or
reduced crossed product of $\mathcal{A}$ by the actions of $G.$ If there is a
onto *-homomorphism $\pi:\mathcal{A\rtimes}_{\alpha}G\longrightarrow
\mathcal{B}$, then
\[
\mathfrak{K}_{3}^{(top)}\left(  \mathcal{B}\right)  \leq\mathfrak{K}%
_{3}^{(top)}\left(  \mathcal{A}\right)  .
\]

\end{corollary}

\begin{proof}
Since $\mathcal{A}$ has no finite-dimensional representations, then
$\pi\left(  \mathcal{A}\right)  $ has no finite-dimensional representations.
Therefore, we can find an element $\pi\left(  a\right)  \in\pi\left(
\mathcal{A}\right)  $ with $\sigma\left(  \pi\left(  a\right)  \right)
=[0,1].$ Note that $\pi\left(  g^{-1}\right)  \pi\left(  a\right)  \pi\left(
g\right)  \subseteq\pi\left(  \mathcal{A}\right)  ,$ then by Theorem \ref{5},
$\mathfrak{K}_{3}^{(top)}\left(  \pi\left(  \mathcal{A}\right)  \cup\left\{
\pi\left(  g\right)  \right\}  \right)  =0.$ From Theorem \ref{6}, we know
that
\[
\mathfrak{K}_{3}^{(top)}\left(  \pi\left(  \mathcal{A}\right)  \cup\left\{
\pi\left(  g_{1}\right)  \right\}  \cup\left\{  \pi\left(  g_{2}\right)
\right\}  \right)  =0.
\]
Let
\[
\mathcal{B}_{n}=C^{\ast}\left(  \pi\left(  \mathcal{A}\right)  \cup\left\{
\pi\left(  g_{1}\right)  \right\}  \cup\cdots\cup\left\{  \pi\left(
g_{n}\right)  \right\}  \right)  .
\]
Then $\mathfrak{K}_{3}^{(top)}\left(  \mathcal{B}_{n}\right)  =0.$

Therefore
\[
\mathfrak{K}_{3}^{(top)}\left(  \mathcal{B}\right)  =\lim\inf_{n}%
\mathfrak{K}_{3}^{(top)}\left(  \mathcal{B}_{n}\right)  =0
\]
by the fact that $\mathcal{B=}\overline{\mathcal{\cup B}_{n}}^{\left\Vert
\cdot\right\Vert }$ and Corollary \ref{7}.
\end{proof}

\section{Applications to Central Sequence Algebras}

In \cite{WS1} Weihua Li and Junhao Shen proved that a separable approximately
divisible C*-algebra $\mathcal{A}$ is singly generated and that if
$\mathcal{A}=C^{\ast}\left(  x_{1},\ldots,x_{n}\right)  $ is an MF-algebra,
then%
\[
\delta^{\text{\textrm{top}}}\left(  x_{1},\ldots,x_{n}\right)  =1.
\]
Later Li and Shen \cite{WS} proved that if $\mathcal{A}$ is an approximately
divisible unital C*-algebra, then Pisier's similarity degree $d\left(
\mathcal{A}\right)  $ is at most $5$. After that Don Hadwin and Weihua Li
\cite{HW2} defined the larger class of weakly approximately divisible
C*-algebras and proved for these algebras the similarity degree is at most
$5$. More recently, Wenhua Qian and Junhao Shen \cite{QS} defined the still
larger class of C*-algebras with property c*-$\Gamma$ and proved that the
similarity degree is at most $3$.

In this section we want to view these results in terms of the central sequence
algebra of a separable unital C*-algebra $\mathcal{A}$. Let $c_{\omega}\left(
\mathcal{A}\right)  $ denote the closed two-sided ideal of the C*-algebra
$l^{\mathcal{1}}\left(  \mathcal{A}\right)  $ given by
\[
c_{\omega}\left(  \mathcal{A}\right)  =\left\{  \left(  a_{n}\right)
_{n\geq1}\in l^{\mathcal{1}}\left(  \mathcal{A}\right)  |\lim
_{n\longrightarrow\omega}\left\Vert a_{n}\right\Vert =0\right\}  .
\]
The ultrapower $\mathcal{A}_{\omega}$ is defined to be the quotient C*-algebra
$l^{\mathcal{1}}\left(  \mathcal{A}\right)  /c_{\omega}\left(  \mathcal{A}%
\right)  ,$ and we denote by $\pi_{\omega}$ the quotient mapping
$l^{\mathcal{1}}\left(  \mathcal{A}\right)  \longrightarrow\mathcal{A}%
_{\omega}.$ Let $l:\mathcal{A}\longrightarrow l^{\mathcal{1}}\left(
\mathcal{A}\right)  $ denote the "diagonal" inclusion mapping $l\left(
a\right)  =\left(  a,a,\cdots\right)  \in l^{\mathcal{1}}\left(
\mathcal{A}\right)  ,$ $a\in\mathcal{A}$; and put $l_{\omega}=\pi_{\omega
}\circ l:\mathcal{A}\longrightarrow\mathcal{A}_{\omega}.$ Both mapping $l$ and
$l_{\omega}$ are injective. If we view $\mathcal{A}$ as a subalgebra of
$\mathcal{A}_{\omega},$ then the relative commutant is defined by
$\mathcal{A}_{\omega}\cap\mathcal{A}^{\prime}$ which is called a central
sequence algebra of $\mathcal{A}$. Suppose $\tau$ is a tracial state on
$\mathcal{A}$ and $\mathcal{N}$ is the weak closure of $\mathcal{A}$ in the
GNS representation determined by $\tau$. The algebra $\mathcal{N}^{\omega
}:=l^{\mathcal{1}}\left(  \mathcal{N}\right)  /c_{\tau,\omega}\left(
\mathcal{N}\right)  $ (with $c_{\tau,\omega}\left(  \mathcal{N}\right)  $ the
bounded sequences in $\mathcal{N}$ with $\lim_{\omega}\left\Vert a\right\Vert
_{2,\tau}=0)$ is a W*-algebra when $\omega$ is a free ultrafilter.

If $\mathcal{M}$ is a $II_{1}$ factor, then $\mathcal{M}$ has \emph{property
}$\Gamma$ if and only if $\mathcal{M}^{\omega}\cap\mathcal{M}^{\prime}$ has a
representing sequence $\left(  U_{1},U_{2},\ldots\right)  $ such that each
$U_{n}$ is a Haar unitary element of $\mathcal{M}$ (i.e., $\tau\left(
U^{n}\right)  =0$ for all $n\in\mathbb{N}$). If $\mathcal{M}$ is a $II_{1}$
von Neumann algebra with a separable predual, then $\mathcal{M}$ is defined in
\cite{QS} to have \emph{property} $\Gamma$ if and only if each $II_{1}$ factor
in the central decomposition of $\mathcal{M}$ has property $\Gamma$. It
follows from direct integral theory that if $\mathcal{M}$ has property
$\Gamma$, then $\mathcal{M}^{\omega}\cap\mathcal{M}^{\prime}$ contains a
representing sequence of Haar unitaries. The following Theorem duo to Dixmier
\cite{Dix} and Connes \cite{Connes}.

\begin{theorem}
\label{150}Let $\mathcal{M}$ be a separable II$_{1}$ factor. The following
conditions are equivalent:

\begin{enumerate}
\item $\mathcal{M}$ has property $\Gamma;$

\item $\mathcal{M}^{\omega}\cap\mathcal{M}^{\prime}\neq\mathbb{C}I;$

\item $\mathcal{M}^{\omega}\cap\mathcal{M}^{\prime}$ is a diffuse von Neumann algebra.
\end{enumerate}
\end{theorem}

let $\pi_{\tau}\left(  \mathcal{A}\right)  ^{\prime\prime}$ be the weak
closure of $\mathcal{A}$ under the GNS representation of $\mathcal{A}$ with
respect to the state $\tau.$

In \cite{QS}, a separable unital C*-algebra is said to have \emph{property
c*-}$\Gamma$ if, for every tracial state $\tau$ on $\mathcal{A}$ such that
$\pi_{\tau}\left(  \mathcal{A}\right)  ^{\prime\prime}$ is a $II_{1}$ factor,
$\pi_{\tau}\left(  \mathcal{A}\right)  ^{\prime\prime}$ has property $\Gamma$,
which is equivalent to $\pi_{\tau}\left(  \mathcal{A}\right)  ^{\prime\prime}$
having property $\Gamma$ whenever $\pi_{\tau}\left(  \mathcal{A}\right)
^{\prime\prime}$ is a $II_{1}$ von Neumann algebra. If $\mathcal{A}$ has no
finite-dimensional representations, then so does $\pi_{\tau}\left(
\mathcal{A}\right)  ^{\prime\prime}$. Therefore $\pi_{\tau}\left(
\mathcal{A}\right)  ^{\prime\prime}$ is $II_{1}$ for every tracial state
$\tau$ on $\mathcal{A}$. So if $\mathcal{A}$ has no finite-dimensional
representations and property c*-$\Gamma,$ then $\pi_{\tau}\left(
\mathcal{A}\right)  ^{\prime\prime}$ has property $\Gamma$ for every tracial
state $\tau$ on $\mathcal{A}$. Actually we can say more in this case.

\begin{lemma}
\label{99}Suppose $\mathcal{A}$ is a separable unital C*-algebra having no
finite-dimensional representations. Then $\mathcal{A}$ has property
c*-$\Gamma$ if and only if for every tracial state $\tau$ on $\mathcal{A}$,
the central sequence algebra of $\pi_{\tau}\left(  \mathcal{A}\right)
^{\prime\prime}$ has no finite-dimensional representations.
\end{lemma}

\begin{proof}
If $\mathcal{A}$ has property c*-$\Gamma$ , then $\pi_{\tau}\left(
\mathcal{A}\right)  ^{\prime\prime}$ has property $\Gamma$ for every tracial
state $\tau$ on $\mathcal{A}$. So the central sequence algebra of each
$II_{1}$ factor in the central decomposition of $\pi_{\tau}\left(
\mathcal{A}\right)  ^{\prime\prime}$ has no finite-dimensional representations
by Theorem \ref{150}.

On the other hand, if the central sequence algebra of $\pi_{\tau}\left(
\mathcal{A}\right)  ^{\prime\prime}$ has no finite-dimensional representations
for every tracial state, then the central sequence algebra of each $II_{1}$
factor in the central decomposition of $\pi_{\tau}\left(  \mathcal{A}\right)
^{\prime\prime}$ has no finite-dimensional representations. So by Theorem
\ref{150}, $\mathcal{A}$ has property c*-$\Gamma$
\end{proof}

\begin{remark}
\label{98}Let $\mathcal{A}$ be a separable unital C*-algebra, let $\tau$ be a
tracial state on $\mathcal{A}$, let $\mathcal{N}$ be the weak closure of
$\mathcal{A}$ under the GNS representation of $\mathcal{A}$ with respect to
the state $\tau,$ and let $\omega$ be a free ultrafilter on $\mathbb{N}$. It
follows that there are the following two natural *-homomorphisms
\[
\mathcal{A}_{\omega}\longrightarrow\mathcal{N}^{\omega},\text{ \ \ \ \ \ }%
\mathcal{A}_{\omega}\cap\mathcal{A}^{\prime}\longrightarrow\mathcal{N}%
^{\omega}\cap\mathcal{N}^{\prime}.
\]

\end{remark}

\begin{lemma}
(\cite{KR1})\label{95}Let $\mathcal{A}$ be a separable unital C*-algebra, let
$\tau$ be a faithful tracial state on $\mathcal{A}$, let $\mathcal{N}$ be the
weak closure of $\mathcal{A}$ under the GNS representation of $\mathcal{A}$
with respect to the state $\tau,$ and let $\omega$ be a free ultrafilter on
$\mathbb{N}$. It follows that the natural *-homomorphisms%
\[
\mathcal{A}_{\omega}\longrightarrow\mathcal{N}^{\omega},\text{ \ \ \ \ \ }%
\mathcal{A}_{\omega}\cap\mathcal{A}^{\prime}\longrightarrow\mathcal{N}%
^{\omega}\cap\mathcal{N}^{\prime}%
\]
are surjective.
\end{lemma}

We say that an MF algebra $\mathcal{A}$ with no finite-dimensional
representations has property MF-c*-$\Gamma$ if, for every MF-trace $\tau$ on
$\mathcal{A}$, the central sequence algebra $\left(  \pi_{\tau}\left(
\mathcal{A}\right)  ^{\prime\prime}\right)  ^{\omega}\cap\pi_{\tau}\left(
\mathcal{A}\right)  ^{\prime}$ has no finite-dimensional representations,
i.e., $\pi_{\tau}\left(  \mathcal{A}\right)  ^{\prime\prime}$ has property
$\Gamma.$

\begin{theorem}
\label{97}(\cite{HW}) If $\mathcal{M}$ is a von Neumann algebra with a central
net of Haar unitaries, then $\mathfrak{K}_{3}\left(  \mathcal{M}\right)  =0.$
\end{theorem}

\begin{theorem}
\label{11}Let $\mathcal{A}$ be a unital MF C*-algebra with no
finite-dimensional representations. If $\mathcal{A}$ has property c*-$\Gamma$,
then $\mathfrak{K}_{3}^{\left(  top\right)  }\left(  \mathcal{A}\right)  =0.$
\end{theorem}

\begin{proof}
Let $\mathcal{N=}\pi_{\tau}\left(  \mathcal{A}\right)  ^{\prime\prime}$ be the
weak closure of $\mathcal{A}$ under the GNS representation of $\mathcal{A}$
with respect to the tracial state $\tau.$ Since $\mathcal{A}$ has property
c*-$\Gamma$, then there is a central sequence $\left\{  u_{n}\right\}  $ of
Haar unitaries in $\mathcal{N}$ such that $\left[  \left\{  u_{n}\right\}
\right]  =u\in\mathcal{N}^{\omega}\cap\mathcal{N}^{\prime}.$ If follows that
$\mathfrak{K}_{3}\left(  \mathcal{A};\tau\right)  =0$ by Theorem \ref{97}.
Hence $\mathfrak{K}_{3}^{\left(  3\right)  }\left(  \mathcal{A};\tau\right)
=0$ by Remark \ref{88}. Since
\[
\mathfrak{K}_{top}^{(3)}(\mathcal{A})\leq\sup_{\tau\in TS\left(
\mathcal{A}\right)  }\mathfrak{K}_{3}^{\left(  3\right)  }\left(
\mathcal{A};\tau\right)  =0,
\]
by Theorem \ref{k3top}, then $\mathfrak{K}_{top}^{(3)}\left(  \mathcal{A}%
\right)  =0.$
\end{proof}

\begin{corollary}
\label{106}Let $\mathcal{A}$ be an MF algebra with no finite-dimensional
representations. Suppose each tracial state on $\mathcal{A}$ is faithful$.$ If
$\mathcal{A}_{\omega}\cap\mathcal{A}^{\prime}$ has no finite-dimensional
representations, then $\mathfrak{K}_{top}^{(3)}\left(  \mathcal{A}\right)
=0.$
\end{corollary}

\begin{proof}
let $\mathcal{N}$ be the weak closure of $\mathcal{A}$ under the GNS
representation of $\mathcal{A}$ with respect to the tracial state $\tau.$
Since $\tau$ is faithful, the natural *-homomorphisms
\[
\mathcal{A}_{\omega}\cap\mathcal{A}^{\prime}\longrightarrow\mathcal{N}%
^{\omega}\cap\mathcal{N}^{\prime}%
\]
is surjective by Lemma \ref{95}. It follows that $\mathcal{N}^{\omega}%
\cap\mathcal{N}^{\prime}$ has no finite-dimensional representation, hence
$\mathcal{A}$ has property c*-$\Gamma$ by Lemma \ref{99}. Therefore
$\mathfrak{K}_{top}^{(3)}\left(  \mathcal{A}\right)  =0$ by Theorem \ref{11}.
\end{proof}

\begin{corollary}
Let $\mathcal{A}=C^{\ast}\left(  x_{1},x_{2},\ldots,x_{n}\right)  $ be a
unital MF C*-algebra with no finite-dimensional representations. If
$\mathcal{A}$ has property MF-c*-$\Gamma$, then $\mathfrak{K}_{top}%
^{(3)}\left(  \mathcal{A}\right)  =0.$
\end{corollary}

\begin{proof}
Let $\mathcal{N=}\pi_{\tau}\left(  \mathcal{A}\right)  ^{\prime\prime}$ be the
weak closure of $\mathcal{A}$ under the GNS representation of $\mathcal{A}$
with respect to the tracial state $\tau.$ Since $\mathcal{A}$ has property
MF-c*-$\Gamma$, $\mathcal{N}^{\omega}\cap\mathcal{N}^{\prime}$ has no
finite-dimensional representation. Then there is a central sequence $\left\{
u_{n}\right\}  $ of Haar unitaries in $\mathcal{N}$ such that $\left[
\left\{  u_{n}\right\}  \right]  =u\in\mathcal{N}^{\omega}\cap\mathcal{N}%
^{\prime}.$ If follows that $\mathfrak{K}_{3}\left(  \mathcal{A};\tau\right)
=0$ by Theorem \ref{97}. Hence $\mathfrak{K}_{3}^{\left(  3\right)  }\left(
\mathcal{A};\tau\right)  =0$ by Remark \ref{88}. It implies that
$\mathfrak{K}_{2}^{\left(  2\right)  }\left(  \mathcal{A};\tau\right)  =0$ for
every MF tracial state $\tau.$ Note
\[
\mathfrak{K}_{top}^{(2)}(\mathcal{A})\leq\sup_{\tau\in\mathcal{T}_{MF}\left(
\mathcal{A}\right)  }\mathfrak{K}_{2}^{\left(  2\right)  }\left(
\mathcal{A};\tau\right)  =0,
\]
by Theorem \ref{li}, then $\mathfrak{K}_{top}^{(3)}\left(  \mathcal{A}\right)
=0$ by Remark \ref{li3}.
\end{proof}

\begin{remark}
We don't know whether property MF-c*-$\Gamma$ is equivalent to $\mathfrak{K}%
_{top}^{(3)}\left(  \mathcal{A}\right)  =0$ in which $\mathcal{A}=C^{\ast
}\left(  x_{1},x_{2},\ldots,x_{n}\right)  $ has no finite-dimensional
representations. But it is well-known that $C_{r}^{\ast}\left(  \mathbb{F}%
_{2}\right)  $ is simple, hence $C_{r}^{\ast}\left(  \mathbb{F}_{2}\right)  $
has no finite-dimensional representation. And it is obvious that $C_{r}^{\ast
}\left(  \mathbb{F}_{2}\right)  $ has no property MF-c*-$\Gamma,$ so we may
hope $\mathfrak{K}_{top}^{(3)}\left(  C_{r}^{\ast}\left(  \mathbb{F}%
_{2}\right)  \right)  =\mathcal{1}$, i.e.,$\mathfrak{K}_{top}^{(2)}\left(
C_{r}^{\ast}\left(  \mathbb{F}_{2}\right)  \right)  \neq0.$ Actually,
Voiculescu \cite{DV} proved that $\delta_{top}\left(  S_{1},S_{2}\right)  =2$,
where $S_{1}$ and $S_{2}$ are free semicircle elements. Therefore
$\mathfrak{K}_{top}^{(2)}\left(  C_{r}^{\ast}\left(  \mathbb{F}_{2}\right)
\right)  \neq0$ by Theorem 3.1.2 in \cite{HLS}.
\end{remark}


\addcontentsline{toc}{chapter}{\bf
BIBLIOGRAPHY}

\begin{thebibliography}{99}                                                                                               %


\bibitem {BK1}{\small B. Blackadar and E. Kirchberg, Generalized inductive
limits of finite dimensional C*-algebras, Math. Ann. 307 (1997), 343-380.}

\bibitem {BK2}{\small B. Blackadar and E. Kirchberg, Inner quasidiagonality
and strong NF algebras, Pacific Journal of mathematics 198 (2001), 307--329.}

\bibitem {BK3}{\small B. Blackadar and E. Kirchberg, Irreducible
representations of inner quasidiagonal C*-algebras, Canadian Mathematical
Bulletin, 54(2011), 385-395.}

\bibitem {BO}{\small N. Brown and N. Ozawa, Algebras and Finite-Dimensional
Approximations, American Mathematical Society (March 12, 2008).}

\bibitem {Connes}{\small A. Connes, Classification of injective factors. Case
}$II_{1},II_{\mathcal{1}},III_{\lambda},\lambda\neq1.${\small Ann. Math. (2),
104(1976),73-115.}

\bibitem {D}{\small K. Davidson, C*-algebras by Example, Amer Mathematical
Society (June 1996).}

\bibitem {Dix}{\small J. Dixmier, Quelques propri}$\acute{e}${\small t}%
$\acute{e}${\small s des suites centrales dans les facteurs de type II}$_{1}%
.${\small  Invent. Math.,7(1969),215-225.}

\bibitem {DH}{\small M. Dost\'{a}l, D. Hadwin, An Alternative to Free Entropy
for Free Group Factors, Acta Mathematica Sinica 3(2003), 419-472.}

\bibitem {GH}{\small L. Ge and D. Hadwin, Ultraproducts of C*-algebras. In:
Recent advances in operator theory and related topics (Szeged, 1999), Oper.
Theory Adv. Appl., 127, Birkh\"{a}user, Basel, 2001, 305-326.}

\bibitem {H}{\small D. Hadwin, Free entropy and approximate equivalence in von
Neumann algebras. In: Operator Algebras and Operator Theory. Contemp Math,
228. Providence, RI:Amer Math Soc, 1998, 111-131}

\bibitem {HLS}{\small D. Hadwin, Q. Li and J. Shen, Topological free entropy
dimensions in Nuclear C*-algebras and in Full Free Products of Unital
C*-algebras, Canadian Journal of Mathematics, 63(2011), 551-590. }

\bibitem {HW}{\small D. Hadwin; W. Li, A Modified Version of Free
Orbit-Dimension of von Neumann Algebras,} {\small arXiv:0801.0825.}

\bibitem {HW2}{\small D. Hadwin; W. Li, The Similarity Degree of Some
C*-algebras, Bulletin of the Australian Mathematical Society 89, 60-69.}

\bibitem {HKM}{\small D. Hadwin, L. Kaonga, Mathes, Ben, Noncommutative
continuous functions, J. Korean Math. Soc. 40 (2003), no. 5, 789--830.}

\bibitem {HS2}{\small D. Hadwin, J. Shen, Topological free entropy dimension
in unital C*-algebra. J. Funct. Anal. 256 (2009), no.7, 2027-2068. }

\bibitem {HS1}{\small D. Hadwin; J. Shen, Free orbit dimension of finite von
Neumann algebras. J. Funct. Anal. 249 (2007), no. 1, 75--91.}

\bibitem {KR}{\small R. Kadison, J. Ringrose, Fundamentals of the Operator
Algebras (Academic, Orlando, FL), (1983, 1986)Vols. 1 and 2.}

\bibitem {HLSW}{\small Q. Li, D. Hadwin, J. Shen, W. Li, MF-traces and Lower
Bound for the Topological Free Entropy Dimension in Unital C*-algebras,
Journal of Operators Theory, 71(2014),15-44. }

\bibitem {WS1}{\small W. Li, J. Shen, Topological Free Entropy Dimension for
Approximately Divisible C*-algebras, Rocky Mountain Journal of Mathematics,
44(6)(2014),1961-1986}

\bibitem {WS}{\small W. Li, J. Shen, The Similarity Degree of Approximately
Divisible C*-algebras, Operators and Matrices, 7(2)2013,425-430.}

\bibitem {QS}{\small W. Qian, J. Shen, Similarity degree of a class of
C*-algebras,\qquad arXiv:1407.1348.}

\bibitem {DM}{\small D. McDuff, Central sequences and the hyperfinite factor,
Proc. London Math. Soc. (3) 21 1970 443--461.}

\bibitem {KR1}{\small E. Kirchberg, M. Rordam, Central Sequence C*-algebras
and Tensorial Absorption of The Jiang-Su Algebra,\qquad arXiv:1209.5311}

\bibitem {DV}{\small D. Voiculescu, The topological version of free
entropy,\ Lett. Math. Phys. 62 (2002), no. 1, 71--82.}

\bibitem {DV1}{\small D. Voiculescu, The analogues of entropy and of Fisher's
information measure in free probability theory II, Invent. Math., 118 (1994),
411-440.}

\bibitem {DV2}{\small D. Voiculescu, The analogues of entropy and of Fisher's
information measure in free probability theory III:The absence of Cartan
subalgebras,\ Geom. Funct. Anal. 6 (1996) 172--199.}
\end{thebibliography}
\end{document}